\documentclass{amsart}

\usepackage{amsmath}
\usepackage{amssymb}
\usepackage{latexsym}
\usepackage{amsopn}
\usepackage{amsthm}

\usepackage[hidelinks]{hyperref}

\newcommand{\bbN}{{\mathbb N}}
\newcommand{\bbR}{{\mathbb R}}

\newcommand{\bbZ}{{\mathbb Z}}
\newcommand{\bbE}{{\mathbb E}}
\newcommand{\bbF}{{\mathbb F}}
\newcommand{\bbP}{{\mathbb P}}

\newcommand{\EE}{\mathrel{E}}
\newcommand{\FF}{\mathrel{F}}

\newcommand*{\ideal}[1][I]{\mathcal{#1}}

\def\underTilde#1{{\baselineskip=0pt\vtop{\hbox{$#1$}\hbox{$\sim$}}}{}}

\def\bSigma{\underTilde{\Sigma}}
\def\bPi{\underTilde{\Pi}}
\def\bDelta{\underTilde{\Delta}}

\newtheorem{thm}{Theorem}[section]
\newtheorem*{thm*}{Theorem}
\newtheorem{cor}[thm]{Corollary}
\newtheorem{prop}[thm]{Proposition}
\newtheorem{lem}[thm]{Lemma}
\newtheorem*{lem*}{Lemma}
\newtheorem{question}{Question}[section]
\newtheorem*{conjecture}{Conjecture}

\newtheorem{mainthm}{Theorem}

\theoremstyle{definition}

\newtheorem{maindfn}{Definition}
\newtheorem{dfn}[thm]{Definition}

\begin{document}

\title[Relative primeness and Borel partition properties]{Relative primeness and Borel partition properties for equivalence relations}
\author{John D. Clemens}
\address{Boise State University, 1910 University Dr., Boise, ID 83725}
\email{johnclemens@boisestate,edu}
\date{}
\thanks{The author would like to thank Jared Holshouser, Ben Miller, Marcin Sabok,  Douglas Ulrich, and Jindra Zapletal for many helpful conversations about this work.}

\begin{abstract}
We introduce a notion of relative primeness for equivalence relations, strengthening the notion of non-reducibility, and show for many standard benchmark equivalence relations that non-reducibility may be strengthened to relative primeness. We introduce several analogues of cardinal properties for Borel equivalence relations, including the notion of a prime equivalence relation and Borel partition properties on quotient spaces. In particular, we introduce a notion of Borel weak compactness, and characterize partition properties for the equivalence relations ${\mathbb F}_2$ and ${\mathbb E}_1$.  We also discuss dichotomies related to primeness, and see that many natural questions related to Borel reducibility of equivalence relations may be viewed in the framework of relative primeness and Borel partition properties.
\end{abstract}

\maketitle

\tableofcontents

\section{Introduction}

The theory of Borel reducibility of definable equivalence relations has proved fruitful in analyzing the complexity of classification problems in diverse areas of mathematics. A chief tool in this endeavor has been the identification of various canonical ``benchmark'' equivalence relations to which new ones can be compared. Hence it is important to understand the relationships between these benchmark relations, and to consider what, exactly, makes these examples canonical, e.g., what sort of minimality or robustness properties they exhibit.

In this paper, we introduce a strong form of non-reducibility, called relative primeness, and show that it holds between many of these benchmark equivalence relations. This notion provides motivation for several strong properties of equivalence relations, which we call primeness and Borel weak compactness, which we expect certain canonical equivalence relations should exhibit. These properties may be viewed as Borel analogues of properties of cardinals, such as cofinality, regularity, and weak compactness, extending the idea that Borel reducibility of equivalence relations may be viewed as a comparison of the definable cardinalities of their quotient spaces.
Many natural questions about Borel equivalence relations may be reformulated within this framework, providing new directions for investigation of old questions and suggesting many new ones.

Our first notion is that of \emph{relative primeness}. Recall that a map $\varphi: X \rightarrow Y$ is a \emph{homomorphism} from $E$ to $F$ if whenever $x_1 \EE x_2$ we have $\varphi(x_1) \FF \varphi(x_2)$, and a \emph{reduction} if $x_1 \EE x_2$ iff $\varphi(x_1) \FF \varphi(x_2)$. 

\begin{maindfn}
Let $E$ and $F$ be Borel equivalence relations on $X$ and $Y$, respectively. We say that $E$ is \emph{prime to} $F$ if, whenever $\varphi$ is a Borel homomorphism from $E$ to $F$, there is a Borel reduction $\rho$ from $E$ to $E$ such that the range of $\varphi \circ \rho$ is contained in a single $F$-class, i.e., $E \leq_B E \upharpoonright \varphi^{-1}[y]_F$ for some $y$.
\end{maindfn}

We will establish that, for many canonical equivalence relations, non-reducibility can in fact be strengthened to relative primeness:

\begin{mainthm}
Let $\mathcal{E}$ be the following collection of benchmark equivalence relations: $\Delta(2)$, $\Delta(\omega)$, $\Delta(\bbR)$, $\bbE_0$, $\bbE_1$, $\bbE_0^{\omega}$, $\bbE_1^{\omega}$, $\bbF_2$, and $\bbE_2$, and let $\mathcal{F}$ include every equivalence relation in $\mathcal{E}$ together with: $E_{\infty}$, $E_{\infty}^{\omega}$, $\cong_{\text{graph}}$, $E_{K_{\sigma}}$, and the universal $E_G^X$.
Then for any $E$ in $\mathcal{E}$ and for any $F$ in $\mathcal{F}$,  either $E \leq_B F$ or $E$ is prime to $F$. 
\end{mainthm}

We next introduce several global properties of equivalence relations:

\begin{maindfn}
We say that a Borel equivalence relation $E$ is \emph{prime} if it has at least two equivalence classes, and for any Borel equivalence relation $F$, either $E \leq_B F$ or $E$ is prime to $F$. We say that $E$ is \emph{regular} if $E$ is prime to any $F$ with $F <_B E$.
\end{maindfn}

Many previous results may be interpreted as showing regularity and primeness of certain benchmark equivalence relations, e.g., $\Delta(\bbR)$, $\bbE_0$, and $\bbF_2$ are prime, and $\bbE_1$, $\bbE_0^{\omega}$, and $\bbE_2$ are regular. We explore properties of prime equivalence relations and investigate approaches to further examples. We also introduce Borel partition properties and consider a strengthening of primeness:

\begin{maindfn}
We say that a Borel equivalence relation $E$ is \emph{Borel weakly compact} if for every symmetric $E \times E$-invariant Borel function $f: X^2 \rightarrow 2$ there is a Borel set $A$ with $E \leq_B E \upharpoonright A$ so that $f$ is constant on $X^2 \setminus E$. 
\end{maindfn}

Previous results establish the Borel weak compactness of $\Delta(\bbR)$ and $\bbE_0$,  but this turns out to be a much rarer property than primeness. We show that it fails for two natural candidates:

\begin{mainthm}
$\bbF_2$ is not Borel weakly compact.
\end{mainthm}

\begin{mainthm}
$\bbE_1$ is not Borel weakly compact.
\end{mainthm}

Primeness of an equivalence relation may be characterized in terms of being reducible to every coarser equivalence relation whose classes are not too large. Although primeness of $\bbE_1$ remains unresolved, we can establish a weaker property for $\bbE_1$ and a large class of equivalence relations induced by ideals:

\begin{mainthm}
Let $\ideal$ be one of the following: $\text{FIN}$,  $\text{FIN} \times 0$, $\ideal_s$, $\ideal_s \times 0$, or a $\bPi_1^1$ ideal of the form $\ideal[J] \times \text{FIN}$ or $\ideal[J] \times \ideal_s \times 0$ or $\ideal[J] \times \text{FIN} \times 0$. Then for any $\bSigma_1^1$ equivalence relation $F$ of countable index over $E_{\ideal}$ we have $E_{\ideal} \sqsubseteq_c F$.
\end{mainthm}

In particular, this holds for $\bbE_1$, $\bbE_0^{\omega}$, $\bbE_1^{\omega}$, and $\bbE_2$. Finally, we show that except low in the Borel reducibilty hierarchy, relative primeness can not be used to characterized potential complexity:

\begin{mainthm}
If $\Gamma$ is a Wadge class containing $\bSigma^0_2$, then there is no equivalence relation $E$ such that for any Borel equivalence relation $F$, $F \in \text{pot}(\Gamma)$ iff $E$ is prime to $F$.
\end{mainthm}

The paper is organized as follows. In \S2 we review basic facts about Borel equivalence relations and introduce the notion of relative primeness together with its fundamental properties. We introduce many standard benchmark equivalence relations in \S3 and establish relative primeness results between them. In \S4 and \S5 we introduce the notion of a prime equivalence relation and stronger versions together with Borel partition properties for quotient spaces, and reformulate previous work in this framework. We discuss the equivalence relation $\bbF_2$ in \S6, exploring primeness and partition properties; we extend this to higher iterates of the Friedman--Stanley jump in \S7. In \S8 we show that $\bbE_1$ is not Borel weakly compact, and establish a weakening of primeness for the equivalence relation $\bbE_1$ and other equivalence relations induced by ideals, and in \S9 we rule out certain global dichotomy theorems related to primeness.

\section{Fundamental facts about relative primeness}

We begin with the fundamental comparisons between definable equivalence relations on Polish spaces.

\begin{dfn}
Given equivalence relations $E$ and $F$ on the Polish spaces $X$ and $Y$, we say that a map $\varphi: X \rightarrow Y$ is a \emph{homomorphism} from $E$ to $F$ if whenever $x_1 \EE x_2$ we have $\varphi(x_1) \FF \varphi(x_2)$, and a \emph{cohomomorphism} if the reverse holds. We say that $\varphi$ is a \emph{reduction} if $x_1 \EE x_2$ iff $\varphi(x_1) \FF \varphi(x_2)$, and $\varphi$ is an \emph{embedding} if it is an injective reduction. We write $E \leq_B F$ when there is a Borel-measurable reduction from $E$ to $F$, $E \sqsubseteq_B F$ when there is a Borel embedding, and $E \leq_c F$ or $E \sqsubseteq_c F$ when there is a continuous reduction or embedding. We write $E \sim_B F$ when $E \leq_B F$ and $F \leq_B E$.
\end{dfn}

A Borel reduction from $E$ to $F$ represents a definable injection from the quotient space $X/E$ into the quotient space $Y/F$. As such, the study of Borel reducibility may be viewed as the study of definable cardinalities. Although we will not do so here, this view may be extended to the study of actual cardinalities in choiceless models of set theory, such as determinacy models. It is thus natural to extend properties of cardinals in ZFC (i.e., initial ordinals) to the realm of definable cardinalities. Much of the focus here may be viewed as trying to characterize large cardinal properties of definable cardinalities, with a view that benchmark equivalence relations should exhibit some sort of transcendence over other equivalence relations. Several previous results discussed below have already taken this viewpoint; here we provide a unifying framework, develop the basic theory, settle several questions, and raise many others for future study.

Our starting point is the following notion.

\begin{dfn}
Let $E$ and $F$ be Borel equivalence relations on $X$ and $Y$, respectively. We say that $E$ is \emph{prime to} $F$ if, whenever $\varphi$ is a Borel homomorphism from $E$ to $F$, there is a Borel reduction $\rho$ from $E$ to $E$ such that the range of $\varphi \circ \rho$ is contained in a single $F$-class, i.e., $E \leq_B E \upharpoonright \varphi^{-1}[y]_F$ for some $y$.
\end{dfn}

That is, there is some $y$ so that $E \upharpoonright \varphi^{-1}[y]_F$ is as complicated as $E$; in this case we may say $\varphi^{-1}[y]_F$ has the same size as $E$. Note that when $E$ has at least two classes this immediately implies that $E \not\leq_B F$, but primeness is generally a stronger notion as it implies that any homomorphism fails to be a reduction in a very strong way. This may also be viewed as saying that when the quotient $X/E$ is partitioned into ``$Y/F$-many pieces'', then at least one of the pieces has effective cardinality of $X/E$ itself. We explore this viewpoint below in terms of Borel partition properties. Note that the trivial equivalence relation $\Delta(1)$ is prime to any $E$. Many results we prove here will actually produce a continuous embedding of $E$ into the preimage of a single $F$-class.

We begin by establishing some basic results about primeness.
We write $\Delta(X)$ for the equality relation on $X$. The equivalence relation $\bbE_0$ is defined on $2^{\omega}$ by $x \mathrel{\bbE_0} y$ iff $x(n)=y(n)$ for all but finitely many $n$.  Given equivalence relations $E$ and $F$ on $X$ and $Y$, we define the product $E \times F$ on $X \times Y$ by $(x_1,y_1) \mathrel{E \times F} (x_2,y_2)$ iff $x_1 \EE x_2 \wedge y_1 \FF y_2$. We define the amalgamation (disjoint union) $E \amalg F$ on $\{0\} \times X \cup \{1\} \times Y$ by $(i,x) \mathrel{E \amalg F} (j,y)$ iff $(i=j=0 \wedge x \EE y) \vee (i=j=1 \wedge x \FF y)$, and similarly for countable amalgamations. We define the countable product $E^{\omega}$ on $X^{\omega}$ by $\bar{x} \mathrel{E^{\omega}} \bar{y}$ iff $x_n \EE y_n$ for all $n$.

We will make use of several Baire category notions. We write $\forall^{\ast} x P(x)$ when $P$ holds for a comeager set of $x$, and $\exists^{\ast} x P(x)$ when $P$  holds for a non-meager set of $x$. We also use the \emph{Vaught transforms}, where for a Polish group $G$ acting on a Polish space $X$ and an open set $U \subseteq G$ and a set $B \subseteq X$ we let $B^{\ast U} = \{ x \in X : \{g \in G : g \cdot x \in B\} \text{ is comeager in $U$} \}$.

\begin{dfn}
We say that $E$ \emph{maintains complexity on comeager (resp. non-meager) sets} if $E \sqsubseteq_c E \upharpoonright C$ for any comeager (resp. non-meager) $C \subseteq X$.
\end{dfn}

When $E$ maintains complexity on comeager sets and is generated by homeomorphisms then $E$ maintains complexity on non-meager sets. A key fact about equivalence relations which maintain complexity on comeager sets is that they achieve their potential descriptive complexity in the following sense. For a Wadge class $\Gamma$, we say that $E$ is \emph{potentially $\Gamma$} if $E$ is Borel reducible to some $F \in \Gamma$. When $E$ maintains complexity on comeager sets and is potentially $\Gamma$ then $E$ is in fact in $\Gamma$, as we may find a comeager set on which the Borel reduction is continuous, and then a continuous embedding of $E$ into this comeager set.
For the same reason, when $E$ maintains complexity on comeager sets it is sufficient to consider only continuous homomorphisms $\varphi$ in the definition of primeness, and the conclusion can be strengthened to include all Baire-measurable $\varphi$. Another immediate consequence is:

\begin{lem}
If $E$ maintains complexity on non-meager sets, then $E$ is prime to $\Delta(\omega)$.
\end{lem}

It is well-known that many of the standard equivalence relations which we study here, such as $\bbE_0$, $\bbE_1$, $\bbE_2$, and $\bbE_0^{\omega}$, maintain complexity on non-meager sets. We will show this for a broad class of equivalence relations induced by ideals in Theorem~\ref{thm:non-meager-sections} below. Kanovei--Sabok--Zpletal have also shown that  $\bbF_2$ maintains complexity on comeager sets, which we show in Corollary~\ref{cor:F2comeager}.
In contrast, a non-hyperfinite countable Borel equivalence relation does not maintain complexity on comeager sets, since any countable Borel equivalence relation is generically hyperfinite.

A related notion to primeness has been previously studied.

\begin{dfn}
We say that $E$ is \emph{generically $F$-ergodic} if, whenever $\varphi$ is a Borel homomorphism from $E$ to $F$, there is $y \in Y$ such that $\varphi^{-1}[y]_F$ is comeager.
\end{dfn}
When $E$ maintains complexity on comeager sets and is generically $F$-ergodic, then $E$ is prime to $F$, but the converse may fail (e.g., when $E$ is $\bbE_1$ and $F$ is $\bbE_0$; see below).
For example, as $\bbE_0$ is generically $\Delta(\bbR)$-ergodic and maintains complexity on comeager sets, we have that $\bbE_0$ is prime to $\Delta(\bbR)$.

We first establish some fundamental properties of primeness.
\begin{lem}
\label{lem:basic}
Let $E$ be prime to $F$.
\begin{enumerate}
\item If $R$ is a Borel equivalence relation with $E \leq_B R \times F$, then $E \leq_B R$.
\item If $F' \leq_B F$ then $E$ is prime to $F'$.
\item If $F' \subseteq F$ and $E$ is prime to $F' \upharpoonright [y]_F$ for all $y \in Y$, then $E$ is prime to $F'$. In particular, this holds when $F$ is of countable index over $F'$ and $E$ maintains complexity on non-meager sets.
\item If $E$ is prime to both $F$ and $F'$, then $E$ is prime to $F \times F'$ and $E$ is prime to $F \cap F'$.
\item If $E$ is prime to both $F$ and $F'$, and $E$ is prime to $\Delta(2)$, then $E$ is prime to $F \amalg F'$. Similarly, if $E$ is prime to each of $F_n$ and $E$ is prime to $\Delta(\omega)$, then $E$ is prime to $\amalg_n F_n$.
\item If $E' \sim_B E$ then $E'$ is prime to $F$.
\end{enumerate}
\end{lem}

\begin{proof}
Assume $E$ is prime to $F$.

\begin{enumerate}
\item 
Let $\varphi$ be a Borel reduction from $E$ to $R \times F$ and write $\varphi(x)=(\varphi_0(x), \varphi_1(x))$. Since $\varphi_1$ is a homomorphism from $E$ to $F$, there is a single $y$ and a Borel reduction $\rho$ from $E$ to $E \upharpoonright \varphi^{-1}[y]_F$. We have that $\varphi \circ \rho$ is still a reduction of $E$ to $R \times F$, so $x_1 \EE x_2$ iff $\varphi_0(\rho(x_1)) \mathrel{R} \varphi_0(\rho(x_2)) \wedge \varphi_1(\rho(x_1)) \FF \varphi_1(\rho(x_2))$. Since $\varphi_1(\rho(x_1)) \FF \varphi_1(\rho(x_2))$ for all $x_1$ and $x_2$, we have that $\varphi_0 \circ \rho$ is a reduction from $E$ to $R$.

\item
Let $\varphi$ be a a homomorphism from $E$ to $F'$, and $\psi$ a reduction from $F'$ to $F$. Then $\psi \circ \varphi$ is a homomorphism from $E$ to $F$, so there is a reduction $\rho$ from $E$ to $E$ such that the range of $\psi \circ \varphi \circ \rho$ is contained in a single $F$-class. Since $\psi$ is a reduction, the range of $\varphi \circ \rho$ must be contained in a single $F'$-class.

\item
Let $\varphi$ be a homomorphism from $E$ to $F'$. Since $F' \subseteq F$, $\varphi$ is also a homomorphism from $E$ to $F$, so there is a reduction $\rho_0$ from $E$ to $E$ such that the range of $\varphi \circ \rho_0$ is contained in a single $F$-class. Then $\varphi \circ \rho_0$ is a homomorphism from $E$ to $F' \upharpoonright [y]_F$, so there is a reduction $\rho_1$ from $E$ to $E$ such that the range of $\varphi \circ \rho_0 \circ \rho_1$ is contained in a single $F'$-class. Then $\rho= \rho_0 \circ \rho_1$ is as desired.

\item
The second part follows immediately from the previous result, using $F \cap F'$ in place of $F'$.
For the first part, let $\varphi$ be a homomorphism from $E$ to $F \times F'$; we may write $\varphi(x) = ( \varphi_0(x), \varphi_1(x))$ where $\varphi_0$ is a homomorphism from $E$ to $F$ and $\varphi_1$ is a homomorphism from $E$ to $F'$. We can then find a reduction $\rho_0$ from $E$ to $E$ so that the range of $\varphi_0 \circ \rho_0$ is contained in a single $F$-class. Then $\varphi_1 \circ \rho_0$ is a homomorphism from $E$ to $F'$, so there is a reduction $\rho_1$ from $E$ to $E$ such that the range of $\varphi_1 \circ \rho_0 \circ \rho_1$ is contained in a single $F'$-class. Since the range of $\varphi_0 \circ \rho_0 \circ \rho_1$ is still contained in a single $F$-class, the range of $\varphi \circ \rho_0 \circ \rho_1$ is contained in a single $F \times F'$-class, so $\rho_0 \circ \rho_1$ is as desired.

\item
We prove the first part; the second is similar. Let $\varphi$ be a homomorphism from $E$ to $F \amalg F'$. Define the homomorphism $\psi$ from $E$ to $\Delta(2)$ by $\psi(x)=0$ if $\varphi(x)$ is in the domain of $F$, and $\psi(x)=1$ if $\varphi(x)$ is in the domain of $F'$. Then $E$ is reducible to the the restriction of $E$ to the preimage of either 0 or 1 under $\psi$, i.e.,  one of $X_0 =\varphi^{-1}[Y]$ or $X_1 =\varphi^{-1}[Y']$.  So there is a reduction $\rho_0$ from $E$ to $E$ whose range is contained in $X_i$ for some $i \in 2$. Then $\varphi \circ \rho_0$ is a homomorphism to either $F$ or $F'$, so there is a reduction $\rho_1$ from $E$ to $E$ with $\varphi \circ \rho_0 \circ \rho_1$ contained in a single $F$-class or a single $F'$-class (and hence a single $F \amalg F'$-class), so $\rho_0 \circ \rho_1$ is as desired. 

\item
Let $\varphi$ be a homomorphism from $E'$ to $F$. If $\psi$ is a reduction of $E$ to $E'$, then $\varphi \circ \psi$ is a homomorphism from $E$ to $F$, so there is $y$ for which $E \leq_B E \upharpoonright (\varphi \circ \psi)^{-1} [y]_F$. But then $E \leq_B E' \upharpoonright \varphi^{-1} [y]_F$, so also $E' \leq_B E' \upharpoonright \varphi^{-1} [y]_F$.
\qedhere
\end{enumerate}
\end{proof}

Even seemingly simple relative primeness results may have non-trivial consequences. For instance, if $E$ is prime to $\Delta(2)$, then for any $E$-invariant Borel set $A$, either $E \leq_B E \upharpoonright A$ or $E \leq_B E \upharpoonright (X \setminus A)$; this property is of particular interest among countable Borel equivalence relations.

The terminology of relative primeness is partially motivated by Property (1), suggestive as it is of the corresponding algebraic property. Property (1) does not imply that $E$ is prime to $F$, though. For example, let $E= \bbE_1 \amalg E_{\infty}$ and $F = \Delta(2)$. Then $E$ is not prime to $F$ since there is a homomorphism sending the domain of $\bbE_1$ to $0$ and the domain of $E_{\infty}$ to $1$, and $E$ is not reducible to either part. Suppose, though, that $E \leq_B R \times F$ for some Borel equivalence relation $R$. As we will see below, both $\bbE_1$ and $E_{\infty}$ are prime to $\Delta(2)$, so that $\bbE_1 \leq_B R$ and $E_{\infty} \leq_B R$. Let $\varphi_1$ and $\varphi_2$ be respective reductions, and note that $\varphi_2$ is countable-to-one. Let $X_1 = \{ x : \exists y \ \varphi_1(x) \mathrel{R} \varphi_2(y)\}$, which is $\bbE_1$-invariant and Borel since it is the projection of a Borel set with countable sections. We cannot have $\bbE_1 \leq_B \bbE_1 \upharpoonright X_1$, since the Lusin--Novikov selection theorem would give a reduction of $\bbE_1 \upharpoonright X_1$ to $E_{\infty}$, and $\bbE_1$ is not essentially countable. Thus $\bbE_1 \leq_B \bbE_1 \upharpoonright (X \setminus X_1)$, so composing with $\varphi_1$ gives a reduction of $\bbE_1$ to $R$ whose range has $R$-saturation disjoint from that of the range of $\varphi_2$. Combining these gives a reduction of $E$ to $R$.

Note that we do not always have $E \amalg F \leq_B R$ when $E \leq_B R$ and $F \leq_B R$. We can ask if the above example extends to other equivalence relations.
\begin{question}
If $E$ and $F$ are relatively prime to one another, $E \leq_B R$, and $F \leq_B R$, is $E \amalg F \leq_B R$?
\end{question}

If $E$ is prime to $F$ then $E$ is prime to $F^n$ for all $n \in \omega$; however, $E$ need not be prime to $F^{\omega}$, e.g., when $E$ is $\bbE_0^{\omega}$ and $F$ is $\bbE_0$. Similarly, if $E$ is prime to $F_n$ for all $n \in \omega$ it need not be prime to $\bigcap_n F_n$, e.g., when $E$ is $\bbE_0^{\omega}$ and $F_n$ is $\bbE_0^n \times I( (2^{\omega})^{\omega})$, where $I( (2^{\omega})^{\omega})= (2^{\omega})^{\omega}\times (2^{\omega})^{\omega}$.

\section{Relative primeness of benchmark equivalence relations}

We begin by showing that many earlier non-reducibility results can be improved to show relative primeness, and establishing that many of the standard benchmark equivalence relations are prime to one another.
We briefly recall the main benchmark equivalence relations which will be used below.

\begin{dfn} Let $X$ be any Polish space.
\begin{enumerate}
\item $\Delta(X)$ is the identity relation on $X$, and $I(X)$ is $X \times X$.
\item $\bbE_0$ is defined on $2^{\omega}$ by $x \mathrel{\bbE_0} y$ iff $\forall^{\infty} n\ x(n)=y(n)$.
\item $\bbE_1$ is defined on $\left(2^{\omega}\right)^{\omega}$ by  $\bar{x} \mathrel{\bbE_1} \bar{y}$ iff $\forall^{\infty} n\ x_n=y_n$.
\item $E_{\infty}$ is the universal countable Borel equivalence relation, which can be represented by the shift action of the free group on two generators, $F_2$, on $2^{F_2}$.
\item For any $E$ on $X$, $E^{\omega}$ is defined on $X^{\omega}$ by $\bar{x} \mathrel{E^{\omega}} \bar{y}$ iff $\forall n\ x_n \EE y_n$. We will consider in particular $\bbE_0^{\omega}$, $\bbE_1^{\omega}$, and $E_{\infty}^{\omega}$.
\item $\bbE_2$ is the summable equivalence relation, defined on $2^{\omega}$ by $x \mathrel{\bbE_2} y$ iff $\sum\{\frac{1}{n+1} : x(n) \neq y(n)\} < \infty$.
\item $\bbE_d$ is the density equivalence relation, defined on $2^{\omega}$ by $x \mathrel{\bbE_d} y$ iff the upper density of $\{n : x(n) \neq  y(n)\}$ is 0.
\item $\bbF_2$ is equality of countable sets of reals, defined on $\left(2^{\omega}\right)^{\omega}$ by $\bar{x} \mathrel{\bbF_2} \bar{y}$ iff $\{x_n : n \in \omega\} = \{y_n : n \in \omega\}$.
\item $\cong_{\text{graph}}$ is the isomorphism relation on countable graphs.
\item $E_{K_{\sigma}}$ is the universal $K_{\sigma}$ equivalence relation.
\item The universal orbit equivalence relation is maximum among all orbit equivalence relations $E_G^X$ induced by the continuous action of a Polish group $G$ on a Polish space $X$.
\end{enumerate}
\end{dfn}

\subsection{${\mathbb E}_1$ and ${\mathbb E}_1^{\omega}$}

We will show that $\bbE_1$ and $\bbE_1^{\omega}$ are prime to any orbit equivalence relation, strengthening a result from \cite{KL}. We will use $E_G^X$ to denote the orbit equivalence relation induced by a continuous action of the Polish group $G$ on the Polish space $X$. We let $F_n$ denote the closed equivalence relation on $(2^{\omega})^{\omega}$ given by $\bar{x} \mathrel{F_n} \bar{y}$ iff $\forall m \geq n (x_m = y_m)$, so that $\bbE_1 = \bigcup_n F_n$. We will also apply $F_n$ to finite sequences sharing the same domain.

We use the following lemma from \cite{Hjorth}:

\begin{lem}[Hjorth]
\label{lem:hjorth}
Let $G$ and $H$ be Polish groups, $\varphi : X \rightarrow Y$ a Baire-measurable homomorphism from $E_G^X$ to $E_H^Y$, and $V$ an open neighborhood of $1_H$. Then for a comeager set of $x \in X$ there is an open neighborhood $U$ of $1_G$ such that
$\forall^{\ast} g \in U  (\varphi (g \cdot x) \in V \cdot \varphi(x))$.
\end{lem}

\begin{thm}
Let $E_G^X$ be any orbit equivalence relation. Then $\bbE_1$ is prime to $E_G^X$.
\end{thm}

\begin{proof}
Let $\varphi$ be a Borel homomorphism from $\bbE_1$ to $E_G^X$; we may assume $\varphi$ is continuous since $\bbE_1$ maintains complexity on comeager sets. We view $2^{\omega}$ as a Polish group under symmetric difference. Let $H_n = (2^{\omega})^n$ act on $(2^{\omega})^{\omega}$ by coordinate-wise symmetric difference, generating  $F_n$, and note that $\varphi$ is a homomorphism from $F_n = E_{H_n}^{(2^{\omega})^{\omega}}$ to $E_G^X$. Let $d_G$ be a complete metric on $G$, and set $V_k = B_{2^{-k}}(1_G)$. By Lemma~\ref{lem:hjorth} we can find for each $n$ and $k$ a comeager set $C_{n,k} \subseteq (2^{\omega})^{\omega}$ such that for all $x \in C_{n,k}$ there is an open neighborhood $1_{H_n} \in U \subseteq H_n$ such that $\forall^{\ast} h \in U \exists  g \in V_k (\varphi(h \cdot x) = g \cdot \varphi(x))$.
 Let $C_0 = \bigcap_{n,k} C_{n,k}$, and for each $k$ let $C_{k+1} = \bigcap_n C_k^{\ast H_n}$, where $C_k^{\ast H_n} = \{x : \forall^{\ast} h \in H_n (h \cdot x \in C_k) \}$. Then each $C_k$ is comeager, and so is $C_{\infty} = \bigcap_k C_k$. Note that for all $x \in C_{\infty}$ and all $n$ we have $\forall^{\ast} h \in H_n (h \cdot x \in C_{\infty})$.

Fix an enumeration $\langle (m_i,j_i) \rangle_i$ of $\omega \times \omega$ and let $A_n = \{ (m_i,j_i) : i < n\}$. For $s \in 2^{A_n}$ and $i \in 2$ we let $s \smallfrown i$ be the obvious element of $2^{A_{n+1}}$. For $n \in \omega$ we will inductively find a strictly increasing sequence $(m_n)_{n \in \omega} \in \omega^{\omega}$, and $x_s \in (2^{\omega})^{\omega}$ and $g_s \in G$ for $s \in 2^{A_n}$ with the following properties:
\begin{enumerate}
\item $x_s \in C_{\infty}$.
\item For $t \in A_m$ with $s \sqsubseteq t$ we have $x_s \upharpoonright (2^{m_n})^{m_n} = x_t \upharpoonright (2^{m_n})^{m_n}$.
\item For $t \in A_n$, if $s \mathrel{F_k} t$  then $x_s \mathrel{F_k} x_t$.
\item For $t \in A_n$, if $\neg s  \mathrel{F_k} t$ then $\neg x_s \upharpoonright (2^{m_n})^{m_n}  \mathrel{F_k} x_t \upharpoonright (2^{m_n})^{m_n}$.
\item  For $i \in 2$ we have $d_G(g_{s \smallfrown i}, g_s) < 2^{-n}$.
\item $g_s \cdot \varphi(x_{\emptyset}) = \varphi(x_s)$.
\end{enumerate}

Granting this, given $x \in (2^{\omega})^{\omega}$ let $s_n = x \upharpoonright A_n$ and set $\rho(x) = \lim_n x_{s_n}$, which exists by condition (2), which also ensures the continuity of $\rho$. Condition (3) and the closedness of the $F_k$'s ensure that $\rho$ is a homomorphism from $\bbE_1$ to $\bbE_1$, whereas conditions (2) and (4) guarantee that it is a cohomomorphism and injective. Thus $\rho$ is a continuous embedding of $\bbE_1$ into itself. To see that the range of $\rho$ is contained in a single $E_G^X$-class, let $x$ be given and set $s_n = x \upharpoonright A_n$. Condition (5) ensures that the sequence $\langle g_{s_n} \rangle_n$ is Cauchy, and hence converges to some $g_{\infty} \in G$. We have $\lim x_{s_n} = \rho(x)$ so $\lim \varphi(x_{s_n}) = \varphi(\rho(x))$ by continuity of $\varphi$. But this limit is the same as $\lim g_{s_n} \cdot \varphi(x_{\emptyset}) = g_{\infty} \cdot \varphi(x_{\emptyset})$ by the continuity of the $G$-action, and hence $\varphi(\rho(x)) \mathrel{E_G^X} \varphi(x_{\emptyset})$ for all $x$.

For the construction:  choose any point $x_{\emptyset} \in C_{\infty}$, and let $m_0=0$ and $g_{\emptyset} = 1_G$. Suppose then $m_n$ and $x_s$ and $g_s$ for $s \in 2^{A_n}$ have been defined to meet the above conditions; we do the same for $n+1$. Let $A_{n+1} \setminus A_n = \{ (m, j) \}$ be the new coordinate on which we extend sequences. Choose $k$ large enough so that for all $s \in 2^{A_n}$ and all $g \in V_k$ we have $d_G(g_s, gg_s) < 2^{-n}$. 
Since each $x_s$ is in $C_0$ and there are only finitely many of them, we can find an open neighborhood $1 _{H_{m+1}} \in U \subseteq H_{m+1}$ such that for all $s \in 2^{A_n}$ we have $\forall^{\ast} h \in U \exists g \in V_k (\varphi(h \cdot x_s) = g \cdot \varphi(x))$; 
since $x_s \in C_{\infty}$ we also have $\forall^{\ast} h \in U (h \cdot x_s \in C_{\infty})$.
We can also take $U$ small enough so that $h \cdot x \upharpoonright (2^{m_n})^{m_n} = x \upharpoonright (2^{m_n})^{m_n}$ for all $x$ and all $h \in U$. Choose an $h$ meeting the previous conditions for all $s \in 2^{A_n}$ and such that $h \notin H_m$ (so $\neg h \cdot x  \mathrel{F_m} x$). Let $x_{s \smallfrown 0} = x_s$, $g_{s \smallfrown 0} = g_s$, $x_{s \smallfrown 1} = h \cdot x_s$, and $g_{s \smallfrown 1} = g'_s g_s$, where $g'_s \in V_k$ satisfies $\varphi(h \cdot x_s) = g'_s \cdot \varphi(x_s)$. Finally, take $m_{n+1}>m_n$ large enough to witness the $F_m$-inequivalence of the relevant pairs. This satisfies the conditions for $n+1$.
\end{proof}

\begin{cor}
$\bbE_1$ is prime to $\Delta(X)$, $\bbE_0$, $\bbE_0^{\omega}$, $E_{\infty}$, $E_{\infty}^{\omega}$, $\bbE_2$, $\bbF_2$, $\bbE_d$, $\cong_{\text{graph}}$, and the universal $E_G^X$.
\end{cor}

The previous theorem can be extended from $\bbE_1$ to $\bbE_1^{\omega}$.
\begin{thm}
Let $E_G^X$ be any orbit equivalence relation. Then $\bbE_1^{\omega}$ is prime to $E_G^X$.
\end{thm}

\begin{proof}
Let $\varphi$ be a continuous homomorphism from $\bbE_1^{\omega}$ to $E_G^X$. Let $H_n^m = (2^{\omega})^n$ act on $((2^{\omega})^{\omega})^{\omega}$ by coordinate-wise symmetric difference on the $m$-th coordinate, generating  $F_n$ on the $m$-th coordinate (i.e., $\prod_{i < m} \Delta((2^{\omega})^{\omega}) \times F_n \times \prod_{i > m} \Delta((2^{\omega})^{\omega})$). Let $d_G$ be a complete metric on $G$, and set $V_k = B_{2^{-k}}(1_G)$. By Lemma~\ref{lem:hjorth} we can find for each $m$, $n$, and $k$ a comeager set $C_{m,n,k} \subseteq ((2^{\omega})^{\omega})^{\omega}$ such that for all $x \in C_{m,n,k}$ there is an open neighborhood $1_{H_n^m} \in U \subseteq H_n^m$ such that $\forall^{\ast} h \in U \exists  g \in V_k (\varphi(h \cdot x) = g \cdot \varphi(x))$.
 Let $C_0 = \bigcap_{m,n,k} C_{m,n,k}$, and for each $k$ let $C_{k+1} = \bigcap_{m,n} C_k^{\ast H_n^m}$. Then each $C_k$ is comeager, and so is $C_{\infty} = \bigcap_k C_k$, and for all $x \in C_{\infty}$ and all $n$ and $m$ we have $\forall^{\ast} h \in H_n^m (h \cdot x \in C_{\infty})$.

Fix an enumeration $\langle (\ell_i,m_i,j_i) \rangle_i$ of $\omega \times \omega \times \omega$ and let $A_n = \{ (\ell_i,m_i,j_i) : i < n\}$. For $s \in 2^{A_n}$ and $i \in 2$ we let $s \smallfrown i$ be the obvious element of $2^{A_{n+1}}$, and we write $s_{\ell}$ for $s \upharpoonright A_n \cap \{\ell\}\times \omega \times \omega$. For $n \in \omega$ we will inductively find a strictly increasing sequence $m_n \in \omega$, and $x_s \in ((2^{\omega})^{\omega})^{\omega}$ and $g_s \in G$ for $s \in 2^{A_n}$ with the following properties:
\begin{enumerate}
\item $x_s \in C_{\infty}$.
\item For $s \sqsubseteq t \in A_m$ we have $x_s \upharpoonright ((2^{m_n})^{m_n})^{m_n} = x_t \upharpoonright ((2^{m_n})^{m_n})^{m_n}$.
\item For $t \in A_n$, if $s_{\ell} \mathrel{F_k} t_{\ell}$  then $(x_s)_{\ell} \mathrel{F_k} (x_t)_{\ell}$.
\item For $t \in A_n$, if $\neg s_{\ell} \mathrel{F_k} t_{\ell}$ then $\neg (x_s)_{\ell} \upharpoonright (2^{m_n})^{m_n} \mathrel{F_k} (x_t)_{\ell} \upharpoonright (2^{m_n})^{m_n}$.
\item  For $i \in 2$ we have $d_G(g_{s \smallfrown i}, g_s) < 2^{-n}$.
\item $g_s \cdot \varphi(x_{\emptyset}) = \varphi(x_s)$.
\end{enumerate}

Granting this, given $x \in ((2^{\omega})^{\omega})^{\omega}$ let $s_n = x \upharpoonright A_n$ and set $\rho(x) = \lim_n x_{s_n}$. As before, $\rho$ is a homomorphism from $\bbE_1^{\omega}$ to $\bbE_1^{\omega}$, as well as a cohomomorphism and injective, so that $\rho$ is a continuous embedding of $\bbE_1^{\omega}$ into itself. To see that the range of $\rho$ is contained in a single $E_G^X$-class, let $x$ be given and set $s_n = x \upharpoonright A_n$. Condition (5) ensures that the sequence $\langle g_{s_n} \rangle_n$ is Cauchy, and hence converges to some $g_{\infty} \in G$. We have $\lim x_{s_n} = \rho(x)$ so $\lim \varphi(x_{s_n}) = \varphi(\rho(x))$ by continuity of $\varphi$. But this limit is the same as $\lim g_{s_n} \cdot \varphi(x_{\emptyset}) = g_{\infty} \cdot \varphi(x_{\emptyset})$ by the continuity of the $G$-action, and hence $\varphi(\rho(x)) \mathrel{E_G^X} \varphi(x_{\emptyset})$ for all $x$.

For the construction,  choose any point $x_{\emptyset} \in C_{\infty}$, and let $m_0=0$ and $g_{\emptyset} = 1_G$. Suppose then $m_n$ and $x_s$ and $g_s$ for $s \in 2^{A_n}$ have been defined to meet the above conditions; we do the same for $n+1$. Let $A_{n+1} \setminus A_n = \{ (\ell,m, j) \}$ be the new coordinate on which we extend sequences. Choose $k$ large enough so that for all $s \in 2^{A_n}$ and all $g \in V_k$ we have $d_G(g_s, gg_s) < 2^{-n}$. 
Since each $x_s \in C_0$ and there are only finitely many of them, we can find an open neighborhood $1 _{H_{m+1}^{\ell}} \in U \subseteq H_{m+1}^{\ell}$ such that for all $s \in 2^{A_n}$ we have $\forall^{\ast} h \in U \exists g \in V_k (\varphi(h \cdot x_s) = g \cdot \varphi(x))$; 
since $x_s \in C_{\infty}$ we also have $\forall^{\ast} h \in U (h \cdot x_s \in C_{\infty})$.
We can also take $U$ small enough so that $h \cdot x \upharpoonright ((2^{m_n})^{m_n})^{m_n} = x \upharpoonright ((2^{m_n})^{m_n})^{m_n}$ for all $x$ and all $h \in U$. Choose an $h$ meeting the previous conditions for all $s \in 2^{A_n}$ and such that $h \notin H_m^{\ell}$ (so $\neg (h \cdot x)_{\ell} \mathrel{F_m} x_{\ell}$). Let $x_{s \smallfrown 0} = x_s$, $g_{s \smallfrown 0} = g_s$, $x_{s \smallfrown 1} = h \cdot x_s$, and $g_{s \smallfrown 1} = g'_s g_s$, where $g'_s \in V_k$ satisfies $\varphi(h \cdot x_s) = g'_s \cdot \varphi(x_s)$. Finally, take $m_{n+1}>m_n$ large enough to witness the $F_m^{\ell}$-inequivalence of the relevant pairs. This satisfies the conditions for $n+1$.
\end{proof}

\begin{cor}
$\bbE_1^{\omega}$ is prime to $\Delta(X)$, $\bbE_0$, $\bbE_0^{\omega}$, $E_{\infty}$, $E_{\infty}^{\omega}$, $\bbE_2$, $\bbF_2$, $\bbE_d$, $\cong_{\text{graph}}$, and the universal $E_G^X$.
\end{cor}

An ostensibly broader class of equivalence relations are the \emph{idealistic} equivalence relations. It is known that $\bbE_1$ is not reducible to any idealistic equivalence relation.

\begin{question}
Are $\bbE_1$ and $\bbE_1^{\omega}$ prime to every idealistic equivalence relation?
\end{question}

\subsection{${\mathbb E}_0^{\omega}$ and ${\mathbb E}_1^{\omega}$}

We next show that many infinite products, such as $\bbE_0^{\omega}$ and $\bbE_1^{\omega}$, are prime to any $F_{\sigma}$ equivalence relation.

\begin{thm}
\label{thm:Fsigma}
Let $E$ be any equivalence relation with an equivalence class $[x_0]_E$ so that $E \leq_B E \upharpoonright \overline{[x_0]_E}$ (e.g., a dense equivalence class), and such that $E^{\omega}$ maintains complexity on comeager sets. Let $F$ be any $F_{\sigma}$ equivalence relation. Then $E^{\omega}$ is prime to $F$.
\end{thm}

\begin{proof}
Let $\varphi$ be a continuous homomorphism from $E^{\omega}$ to $F$. Let $F = \bigcup_n F_n$ where the $F_n$'s are closed, symmetric relations with $F_n \subseteq F_{n+1}$. 
We write $\bar{x}  \mathrel{E^{<k}} \bar{y}$ if $\bar{x} \mathrel{E^{\omega}} \bar{y}$ and $x_i = y_i$ for all $i \geq k$; similarly, we write  $\bar{x}  \mathrel{E^{\geq k}} \bar{y}$ if $\bar{x} \mathrel{E^{\omega}} \bar{y}$ and $x_i = y_i$ for all $i < k$.
We say $\bar{x} \in X^{\omega}$ is \emph{good} if there are $k$ and $\bar{y}$ with $\bar{x} E^{<k} \bar{y}$ such that for any $\bar{z}$ with $\bar{z} \mathrel{E^{\geq k}} \bar{y}$ we have $\varphi(\bar{z}) \mathrel{F_k} \varphi(\bar{x})$; otherwise we say $\bar{x}$ is bad.

We claim that there are no bad points. Let $\bar{x}$ be bad, so that for all $k$ and all $\bar{y}$ with $\bar{x} \mathrel{E^{<k}} \bar{y}$ there is $\bar{z}$ with $\bar{z} \mathrel{E^{\geq k}} \bar{y}$ such that $\neg \varphi (\bar{z}) \mathrel{F_k} \varphi(\bar{x})$. Let $\bar{x}_0 = \bar{x}$ and $k_0=0$. Given $k_n$ and $\bar{x}_n$ such that $\bar{x}_n \mathrel{E^{<k_n}} \bar{x}$, by our assumption of badness for $\bar{x}$ we may find $\bar{y}$ with $\bar{y} \mathrel{E^{\geq k_n}} \bar{x}_n$ such that $\neg \varphi(\bar{y}) \mathrel{F_{k_n}} \varphi(\bar{x})$. Since $F_{k_n}$ is closed we may find $k_{n+1} > k_n$ such that for any $\bar{z}$ with $\bar{z} \mathrel{E^{\geq k_{n+1}}} \bar{y}$ we have $\neg \varphi(\bar{z}) \mathrel{F_{k_n}} \varphi(\bar{x})$. 
Let $\bar{x}_{n+1} = \bar{y} \upharpoonright k_{n+1} \smallfrown \bar{x} \upharpoonright (\omega \setminus k_{n+1})$, so that $\bar{x}_{n+1} \mathrel{E^{\geq k_{n+1}}} \bar{y}$ and $\bar{x}_{n+1} \mathrel{E^{< k_{n+1}}} \bar{x}$. Repeat for all $n \in \omega$.
Then $\bar{x}' = \lim_n \bar{x}_n$ exists since 
$\bar{x}_{n+1} \mathrel{E^{\geq n}} \bar{x}_n$, and satisfies $\bar{x}' E^{\omega} \bar{x}$; however, $\bar{x}' \mathrel{E^{\geq k_{n+1}}} \bar{x}_{n+1}$ for all $n$, so $\neg \varphi(\bar{x}') \mathrel{F_{k_n}} \varphi(\bar{x})$ for all $n$, so $\neg \varphi(\bar{x}') \FF \varphi(\bar{x})$, contradicting that $\varphi$ was a homomorphism.

Now let $x_0 \in X$ be such that $E \leq_B E \upharpoonright C$, where $C = \overline{[x_0]_E}$. 
Let $\psi$ be a reduction from $E$ to $E \upharpoonright C$, so that $\psi$ induces a reduction $\overline{\psi}$ from $E^{\omega}$ to $E^{\omega} \upharpoonright C^{\omega}$. 
Let $\overline{x_0}$ be the sequence with constant value $x_0$.
Since $\overline{x_0}$ is good, we may fix $k$ and $\bar{y}$ with $\overline{x_0} \mathrel{E^{<k}} \bar{y}$ to witness goodness. Since $\varphi$ is continuous and $F_k$ is closed, the set $A=\varphi^{-1}\{y : y \mathrel{F_k} \varphi(\overline{x_0})\}$ is closed and contains $\{ \bar{w} : \bar{w} \mathrel{E^{\geq k}} \bar{y}\}$. This latter set is dense in the closed set $B= \{\bar{y} \upharpoonright k \} \times C^{\omega \setminus k}$ and thus $A$ contains $B$. Hence $\varphi(\bar{w}) \mathrel{F_k} \varphi(\overline{x_0})$ for all $\bar{w} \in B$. 
Define $\rho(\bar{z}) = \bar{y} \upharpoonright k \smallfrown \overline{\psi}(\bar{z})$; $\rho$ is clearly a reduction from $E^{\omega}$ to $E^{\omega}$, and $\varphi(\rho(\bar{z})) \mathrel{F_k} \varphi(\overline{x_0})$ for all $\bar{z}$, which completes the proof. 
\end{proof}

This result  can be generalized to other products $\prod_i E_i$, provided that for each $i$ there are infinitely many $j$ with $E_i \sqsubseteq_c E_j$.

\begin{cor}
Both $\bbE_0^{\omega}$ and $\bbE_1^{\omega}$ are prime to each of $\Delta(X)$, $\bbE_0$, $\bbE_1$, $E_{\infty}$, $\bbE_2$, and $E_{K_{\sigma}}$.
\end{cor}

The above theorem does not apply to $E_{\infty}^{\omega}$, as it does not maintain complexity on comeager sets; in fact there is a comeager set $C$ so that $E_{\infty}^{\omega} \upharpoonright C\sim_B \bbE_0^{\omega}$.

\begin{question}
Is $E_{\infty}^{\omega}$ prime to every $F_{\sigma}$ equivalence relation?
\end{question}

We are not sure if these results can be extended higher in the Borel hierarchy.

\begin{question}
Can we extend Theorem~\ref{thm:Fsigma} result from $F_{\sigma}$ to  $\bSigma^0_3$ equivalence relations? 
\end{question}

There has been some work in determining the precise relationships between $\bbE_1$,  $\bbE_0^{\omega}$ and their disjoint union and products. For instance, Kechris and Louveau have asked if there are any equivalence relations properly between $\bbE_1 \amalg \bbE_0^{\omega}$ and $\bbE_1 \times \bbE_0^{\omega}$.
Since $\bbE_1^{\omega}$ is also prime to $\bbE_1$, using Lemma~\ref{lem:basic} (4) we also get:
\begin{cor}
$\bbE_1^{\omega}$ is prime to $\bbE_1 \times \bbE_0^{\omega}$.
\end{cor}

\subsection{${\mathbb E}_2$ and turbulent actions}

We note here several consequences regarding primeness for $\bbE_2$ which follow immediately from earlier results. The notion of a \emph{turbulent} orbit equivalence relation is introduced in \cite{Hjorth}; we omit some definitions as we do not need the details here. We recall the following:

\begin{thm}
If $E_G^X$ is a turbulent Polish $G$-space and $E_{S_{\infty}}^Y$ is a Polish $S_{\infty}$-space, then $E_G^X$ is generically $E_{S_{\infty}}^Y$-ergodic.
\end{thm}

Since $\bbE_2$ and $\bbE_d$ are turbulent and maintain complexity on comeager sets, we thus have:

\begin{cor}
$\bbE_2$ and $\bbE_d$ are prime to any $E_{S_{\infty}}^Y$. In particular, $\bbE_2$ is prime to $\Delta(X)$, $\bbE_0$, $\bbE_0^{\omega}$, $E_{\infty}$, $E_{\infty}^{\omega}$, $\bbF_2$, and $\cong_{\text{graph}}$.
\end{cor}

Further conclusions can be obtained from results of Kanovei--Reeken from \cite{kan-ree}. 

\begin{dfn}
Let $\mathcal{E}$ be the smallest class of equivalence relations containing $\Delta(X)$ for Polish spaces $X$ and closed under the following operations:
\begin{enumerate}
\item countable unions on the same space (when this yields an equivalence relation)
\item countable intersections on the same space
\item countable amalgamations (disjoint unions)
\item countable products
\item Fubini products mod FIN, where $\bar{x} \mathrel{\prod_i E_i / \text{FIN}} \bar{y}$ iff $\{i : \neg x_i \EE_i y_i \} \in \text{FIN}$
\item the Friedman--Stanley jump $E^{+}$, where $\bar{x} \mathrel{E^{+}} \bar{y}$ iff $\{ [x_i]_E : i \in \omega\} = \{[y_i]_E : i \in \omega \}$
\end{enumerate}
\end{dfn}

In particular, $\bbE_1$ is in $\mathcal{E}$, being the Fubini product of $\Delta(2^{\omega})$ mod FIN, as is $\bbE_1^{\omega}$.

\begin{thm}[Kanovei--Reeken]
Let $E_G^X$ be a turbulent Polish $G$-space, and $F \in \mathcal{E}$. Then $E_G^X$ is generically $F$-ergodic.
\end{thm}

\begin{cor}
$\bbE_2$ and $\bbE_d$ are prime to $\bbE_1$ and prime to $\bbE_1^{\omega}$.
\end{cor}

It is known that $\bbE_2 \not\leq_B \bbE_d$ and $\bbE_d \not\leq_B \bbE_2$, but we do not know if this can be extended to relative primeness.

\begin{question}
Is $\bbE_2$ prime to $\bbE_d$ or vice versa?
\end{question}

\subsection{The universal countable Borel equivalence relation $E_{\infty}$}
\label{subsection:countable}

Several long-standing questions about $E_{\infty}$ may be rephrased in terms of primeness. Martin's Conjecture (MC) on Turing degree-invariant functions, for instance, has several consequences for primeness results concerning $E_{\infty}$. Discussion of Martin's Conjecture, Turing equivalence $\equiv_T$, and countable Borel equivalence relations may be found in \cite{MSS}.
One direct result is the following:

\begin{thm}[Marks, Theorem 3.1 of \cite{MSS}] 
$E_{\infty}$ is prime to $\Delta(\bbR)$.
\end{thm}

However, $E_{\infty}$ is not prime to Turing equivalence $\equiv_T$, and thus is not a prime equivalence relation if $\equiv_T$ is not a universal countable Borel equivalence relation (prime relations will be introduced in the next section). Note that Martin's Conjecture contradicts $\equiv_T$ being a universal countable Borel equivalence relation.

\begin{dfn}
A countable Borel equivalence relation $E$ is \emph{weakly universal} if for every countable Borel equivalence relation $F$ there is a countable-to-one homomorphism from $F$ to $E$. Equivalently, $E$ is weakly universal if it contains a universal countable Borel equivalence relation.
\end{dfn}
In particular, $\equiv_T$ is known to be weakly universal. If $E$ is weakly universal then there is a countable-to-one homomorphism from $E_{\infty}$ to $E$ so that $E_{\infty}$ is not prime to $E$. Hence:

\begin{lem} 
If MC is true, then $\equiv_T <_B E_{\infty}$ but $E_{\infty}$ is not prime to $\equiv_T$.
\end{lem}

Under Martin's Conjecture, weak universality gives a precise criterion for $E_{\infty}$ to be prime to $E$, established by Marks using results of Thomas from \cite{thomas}:

\begin{thm}[Marks, Theorem 3.3 of \cite{MSS}] 
If MC is true and $E$ is countable and not weakly universal, then $E_{\infty}$ is prime to $E$.
\end{thm}

Thus, under MC a countable Borel equivalence relation $E$ is weakly universal if and only if $E_{\infty}$ is not prime to $E$.
A question posed by Hjorth is whether every weakly universal countable Borel equivalence relation is universal. The above observations show that this would follow from primeness results about $E_{\infty}$:

\begin{lem}
If $E_{\infty}$ is prime to every countable $E$ with $E <_B E_{\infty}$, then for any countable $E$ with $E_{\infty} \subseteq E$ we have that $E_{\infty} \sim_B E$.
\end{lem}

\subsection{Summary of benchmarks}

Summarizing the above results, together with the result of Kanovei--Sabok--Zapletal  proved in Section~\ref{sec:f2prime} (which show that for each Borel equivalence relation $E$, either $\bbF_2 \leq_B E$ or $\bbF_2$ is prime to $E$), we have the following:

\begin{thm}
Let $\mathcal{E}$ be the following collection of benchmark equivalence relations: $\Delta(2)$, $\Delta(\omega)$, $\Delta(\bbR)$, $\bbE_0$, $\bbE_1$, $\bbE_0^{\omega}$, $\bbE_1^{\omega}$, $\bbF_2$, and $\bbE_2$, and let $\mathcal{F}$ include every equivalence relation in $\mathcal{E}$ together with: $E_{\infty}$, $E_{\infty}^{\omega}$, $\cong_{\text{graph}}$, $E_{K_{\sigma}}$, and the universal $E_G^X$.
Then for any $E$ in $\mathcal{E}$ and for any $F$ in $\mathcal{F}$,  either $E \leq_B F$ or $E$ is prime to $F$. 
\end{thm}

Notable omissions here are when $E$ is one of  $E_{\infty}$, $E_{\infty}^{\omega}$, $\cong_{\text{graph}}$, $E_{K_{\sigma}}$, or the universal $E_G^X$.

\section{Prime equivalence relations}
\label{sec:nodes}

We now introduce global properties of equivalence relations arising from primeness.
A \emph{node} in the Borel reducibility hierarchy is a Borel equivalence relation $E$ such that for any Borel equivalence relation $F$, either $E \leq_B F$ or $F \leq_B E$. Kechris and Louveau showed in \cite{KL} that the only nodes are $\Delta(n)$ for $n \leq \omega$, $\Delta(\bbR)$, and $\bbE_0$. This limits the possibility for global dichotomies regarding reducibility among Borel equivalence relations, although many additional local dichotomies have been proved (where the collection of equivalence relations is restricted). We can hope for a richer class, and more global dichotomies, if we modify the second alternative.

The following definitions may be viewed as analogues of regular cardinals in ZFC. One formulation of this is to say that a cardinal $\kappa$ is regular if for any $\lambda < \kappa$ and any partition of $\kappa$ into $\lambda$-many subsets, at least one subset must have cardinality $\kappa$. Since definable cardinalities are not linearly ordered, we can consider two possible properties.

\begin{dfn}
We say that a Borel equivalence relation $E$ is \emph{prime} if it has at least two equivalence classes, and for any Borel equivalence relation $F$, either $E \leq_B F$ or $E$ is prime to $F$. We say that $E$ is \emph{regular} if $E$ is prime to any $F$ with $F <_B E$.
\end{dfn}

Any prime equivalence relation is regular. Following algebraic practice we do not consider $\Delta(1)$ to be prime; perhaps we should call it a unit. This ensures that for each $F$, the two possibilities ($E \leq_B F$ or $E$ is prime to $F$) are mutually exclusive. Note that we have restricted ourselves to Borel equivalence relations; we might extend this to analytic equivalence relations $E$, in which case we may wish to consider all analytic equivalence relations $F$. Many of the results here will extend to analytic relations via reflection arguments. We easily have that $\Delta(2)$ is prime, as is $\Delta(\omega)$ by the infinite version of Ramsey's Theorem. Various dichotomy theorems may be rephrased to establish other primeness and regularity results, which we summarize briefly here. Silver's Theorem, for instance, shows that any Borel equivalence relation with uncountably many equivalence classes has a prefect set of equivalence classes, from which we immediately get:

\begin{thm}[Silver, \cite{silver}]
$\Delta(\bbR)$ is prime.
\end{thm}

Similarly, the Generalized Glimm--Effros Dichotomy due to Harrington--Kechris--Louveau established that every Borel equivalence relation is either reducible to $\Delta(\bbR)$ or reduces $\bbE_0$. Using that $\bbE_0$ is prime to $\Delta(\bbR)$, we have:

\begin{thm}[Harrington--Kechris--Louveau, \cite{HKL}] 
$\bbE_0$ is prime.
\end{thm}

We might ask if $\bbE_0$ is actually a node with respect to relative primeness.

\begin{question}
If $\bbE_0 <_B E$, is $E$ prime to $\bbE_0$?
What if $E$ is minimal above $\bbE_0$?
\end{question}

The Kechris--Louveau dichotomy shows that any equivalence relation $E <_B \bbE_1$ is reducible to $\bbE_0$. Using the result from earlier that $\bbE_1$ is prime to $\bbE_0$, this gives:

\begin{thm}[Kechirs--Louveau, \cite{KL}]
$\bbE_1$ is regular.
\end{thm}

In \cite{HK}, Hjorth--Kechris established that any $E <_B \bbE_0^{\omega}$ is reducible to $\bbE_0$, so using the earlier result that $\bbE_0^{\omega}$ is prime to $\bbE_0$ we have:

\begin{thm}[Hjorth--Kechris, \cite{HK}]
$\bbE_0^{\omega}$ is regular.
\end{thm}
In fact, the Seventh Dichotomy Theorem from the same paper established that $\bbE_0^{\omega}$ is either reducible to or prime to any $E$ with $E \leq_B E_G^X$ where $G$ is a closed subgroup of $S_{\infty}$ admitting an invariant metric (Hjorth later extended this to closed $G \subseteqq S_{\infty}$ without an invariant metric). 

A result of Hjorth in \cite{hjorth2} shows that if $E \leq_B \bbE_2$ then either $E \leq_B E_{\infty}$ or $E \sim_B \bbE_2$, so using that $\bbE_2$ is prime to $E_{\infty}$ we have:

\begin{thm}[Hjorth, \cite{hjorth2}]
$\bbE_2$ is regular.
\end{thm}

More recently, Kanovei--Sabok--Zapletal have  established that $\bbF_2$ is prime. We will present a proof of this in Section~\ref{sec:f2prime} below, and use the tools developed there to further analyze $\bbF_2$.
\begin{thm}[Kanovei--Sabok--Zapletal]
$\bbF_2$ is prime.
\end{thm}

 These are the only known examples of regular and prime equivalence relations at present. We may hope to establish that other benchmark equivalence relations have these properties, or that large collections of equivalence relations do. We do not know if the two notions are distinct, although it seems likely that they are.
 
 \begin{question}
 Is every regular Borel equivalence relation prime? In particular, are $\bbE_1$, $\bbE_0^{\omega}$, and $\bbE_2$ prime?
 \end{question}
 
 \begin{question} Is the universal orbit equivalence relation $E_G^X$ prime?
\end{question}
 
We begin by establishing some fundamental properties of prime equivalence relations, which illustrate some aspects of why we might hope for this property among canonical benchmarks.
First we note that Lemma~\ref{lem:basic} (6) shows that primeness is preserved under bireducibility.

\begin{lem}
If $E$ is prime and $F \sim_B E$ then $F$ is prime.
\end{lem}

The following provides an alternate characterization of prime equivalence relations.

\begin{lem}
\label{lem:altprime}
$E$ is prime if and only if $E \leq_B F$ for every Borel $F$ with $E \subseteq F$ such that $E \upharpoonright [x]_F <_B E$ for all $x$.
\end{lem}

\begin{proof}
Suppose $E$ is prime and let $E \subseteq F$, so that the identity is a homomorphism from $E$ to $F$. If $E \upharpoonright [x]_F <_B E$ for all $x$,then there is no embedding $\rho$ from $E$ to $E$ whose range is contained in a single $F$ class, so $E$ is not prime to $F$; hence $E \leq_B F$.

Conversely, suppose $E$ has the given property and $E$ is not prime to $F$. Then there is a homomorphism $\varphi$ from $E$ to $F$ for which there is no embedding of $E$ into the preimage of a single $F$-class. Let $F' = (\varphi \times \varphi)^{-1}[F]$, so $E \subseteq F'$ and $E \upharpoonright [x]_{F'} <_B E$ for all $x$. Then $E \leq_B F'$ by our assumption, and as $F' \leq_B F$ we have $E \leq_B F$.
\end{proof}

The following may again be seen as justification for the term ``prime''.

\begin{lem}
\label{lem:factor}
 If $E$ is prime, then for any two Borel equivalence relations $F_1$ and $F_2$ with $E \leq_B F_1 \times F_2$ we have either $E \leq_B F_1$ or $E \leq_B F_2$.
\end{lem}

\begin{proof}
If $E$ is not reducible to either $F_1$ or $F_2$ then it is prime to both, so by Lemma~\ref{lem:basic} (4) it is prime to $F_1 \times F_2$, contradicting that $E \leq_B F_1 \times F_2$.
\end{proof}

\begin{question}
Is the conclusion of Lemma~\ref{lem:factor} equivalent to primeness? 
\end{question}

We suspect this is not the case, but we do not know any examples with this property other than known prime relations.

Note that Lemma~\ref{lem:factor} cannot be extended to infinite products, as $\Delta(2^{\omega}) \sim_B\Delta(\bbR)$ is prime and  $\Delta(2^{\omega}) = \prod_{ n \in \omega} \Delta(2)$, but $\Delta(\bbR) \not\leq_B \Delta(2)$.

\begin{lem}
If $E$ is prime and $E \leq_B F_1 \cap F_2$ then either $E \leq_B F_1$ or $E \leq_B F_2$.
\end{lem}

\begin{proof}
If $E$ is not reducible to either $F_1$ or $F_2$ then it is prime to both, so by Lemma~\ref{lem:basic} (4) it is prime to $F_1 \cap F_2$, contradicting that $E \leq_B F_1 \cap F_2$.
\end{proof}

This lemma also does not extend to infinite intersections, since $\Delta(2^{\omega})$ is prime and if we let $x \mathrel{F_n} y \ \Leftrightarrow\ x(n)=y(n)$ then $F_n$ has only two classes but $\bigcap_{n \in \omega} F_n = \Delta(2^{\omega})$. As discussed below, Kanovei--Sabok--Zapletal have established this result for $\bbF_2$, though; namely, if $\bbF_2 \leq_B \bigcap_n F_n$ then $\bbF_2 \leq_B F_n$ for some $n$ (Corollary 6.30 of \cite{ksz}).

As discussed in Section \ref{subsection:countable}, primeness of $E_{\infty}$ is closely tied to Martin's conjecture.

\begin{question}
Is $E_{\infty}$ prime? 
\end{question}

\section{Stronger primeness notions and Borel partition properties}

We can consider several properties stronger than primeness.

\begin{dfn}
We say that a Borel equivalence relation $E$ is \emph{uniformly prime} if for every Borel equivalence relation $F$ with $E \subseteq F$, either $E \leq_B E \upharpoonright [x]_F$ for some $x$, or else there is a Borel set $A$ with $E \leq_B E \upharpoonright A$ which is $F \setminus E$-discrete, i.e., if $x_1,x_2 \in A$ with $x_1 \FF x_2$ then $x_1 \EE x_2$.
\end{dfn}

As the second condition implies $E \leq_B F$, Lemma~\ref{lem:altprime} shows that every uniformly prime $E$ is prime. All known examples of prime equivalence relations are in fact uniformly prime, but we do not know if the conditions are equivalent.

\begin{question}
Is every prime equivalence relation uniformly prime?
\end{question}

A stronger property, naturally suggested by the concept of a weakly compact cardinal, is the following:

\begin{dfn}
We say that a Borel equivalence relation $E$ is \emph{Borel weakly compact} if for every symmetric $E \times E$-invariant Borel function $f: X^2 \rightarrow 2$ there is a Borel set $A$ with $E \leq_B E \upharpoonright A$ so that $f$ is constant on $X^2 \setminus E$. 
\end{dfn}

That is, any Borel partition of the set of pairs of distinct $E$ classes has a homogeneous set to which all of $E$ can be reduced. 
Borel weak compactness again implies primeness:

\begin{lem}
If $E$ is Borel weakly compact then $E$ is uniformly prime, and hence prime.
\end{lem}
\begin{proof}
Given $E \subseteq F$, let $f : X^2 \rightarrow 2$ be defined by $f(x,y)=0$ when $x \FF y$ and $1$ if not. This is symmetric and $E \times E$-invariant. A homogeneous set $A$ with constant value 0 gives a reduction of $E$ to a single $F$-class, and one with constant value 1 gives an $F \setminus E$-discrete set.
\end{proof}

We will see below, though, that not every uniformly prime equivalence relation is Borel weakly compact; in particular, this is the case for $\bbF_2$.

Following the usual Erd\"{o}s notion, we can consider various such partition relations.

\begin{dfn}
\label{dfn:partition}
For an equivalence relation $E$ on a Polish space $X$ we use $[E]^n$ to denote the set of pairwise $E$-inequivalent $n$-tuples from $X$. For an equivalence relation $F$ on $Y$, a \emph{Borel $F$-partition of $[E]^n$} is a Borel function $f: [E]^n \rightarrow Y$ which is invariant and symmetric, i.e., if $\{[x_1]_E,\ldots, [x_n]_E\} = \{[x'_1]_E,\ldots, [x'_n]_E\}$ then $f(x_1,\ldots, x_n) \FF f(x'_1,\ldots, x'_n)$.  When $F$ is $\Delta(2)$ we may identify an $F$-partition with a symmetric invariant subset of $[E]^n$.
\end{dfn}

\begin{dfn} For Borel equivalence relations $E$, $F$, and $R$ on Polish spaces $X$, $Y$, and $Z$, respectively, we write $E \rightarrow_B (R)^n_F$ to mean that for every Borel $F$-partition $f: [E]^n \rightarrow Y$ there is a Borel set $A$ so that $R \leq_B E \upharpoonright A$ and $f$ maps $A^n \cap [E]^n$  into a single $F$-class. When $F$ is $\Delta(Y)$ we simply write $E \rightarrow_B (R)^n_Y$. 
\end{dfn}

We call such a set $A$ \emph{homogeneous} for the partition $f$.
We observe that if there is an analytic homogeneous set,  then there is in fact a  Borel $E$-invariant homogeneous set for each Borel partition.

\begin{lem}
\label{lem:reflection}
Let $E$, $F$, and $R$ be Borel equivalence relations, and $f: [E]^n \rightarrow Y$ an $F$-partition. If $A$ is an analytic set which is homogeneous for $f$, then there is a Borel $E$-invariant set $B$ which is homogeneous for $f$ with $A \subseteq B$.
\end{lem}

\begin{proof}
Let $y_0$ be such that $f$ maps all $n$-tuples from distinct $E$-classes of $A$ into $[y_0]_F$.
The statement that $f$ maps all such tuples from a set $A$ into $[y_0]_F$ is $\bPi^1_1$-on-$\bSigma^1_1$, since it holds when
\[ \forall x_1 \cdots \forall x_n \left(  \left(\bigwedge_{1 \leq i < j \leq n} \neg x_i \EE x_j \wedge \bigwedge_{1 \leq i \leq n} A(x_i) \right) \rightarrow f(x_1,\ldots, x_n) \FF y_0 \right) .\]
Thus the First Reflection Theorem (Theorem 35.10 of \cite{Kechris}) gives a Borel set $B_0$  which is homogeneous for $f$ with $A \subseteq B_0$. Let $A_0$ be the $E$-saturation of $B_0$, which is analytic and still homogeneous for $f$. Repeat to find an increasing sequence $A\subseteq B_0 \subseteq A_0 \subseteq B_1 \subseteq  A_1 \subseteq \cdots$ of homogeneous sets with $A_n$ $E$-invariant and $B_n$ Borel; then $B = \bigcup_n B_n$ will be Borel, $E$-invariant, and homogeneous for $f$.
\end{proof} 

As with partition relations on cardinals, such a relation remains true if we replace $E$ by an equivalence relation into which it embeds, and if we replace $F$ or $R$ by equivalence relations embedding into them. 

\begin{lem}
Suppose $E \rightarrow_B (R)^n_F$, $E \leq_B E'$, $F' \leq_B F$, and $R' \leq_B R$. Then $E' \rightarrow_B (R')^n_{F'}$
\end{lem}

\begin{proof}
Let $\rho$ witness that $E \leq_B E'$, and $\varphi$ that $F' \leq_B F$. Suppose $f'$ is an $F'$-partition of $[E']^n$. Then $f'$ composed with (the $n$-fold product of) $\rho$ induces an $F'$-partition of $[E]^n$ and the composition of $\varphi$ with this induces an $F$-partition of $[E]^n$. By assumption there is a set $A$ with $R \leq_B E \upharpoonright A$ so that $f$ maps all $n$-tuples from distinct $E$-classes from $A$ into a single $F$-class. Let $A' = \rho[A]$ (which is analytic), so that $R$, and hence $R'$ is reducible to $E' \upharpoonright A'$. We have that $\varphi \circ f'$ maps $n$-tuples from distinct $E'$-classes from $A'$ into a single $F$-class, and hence $f'$ maps them into a single $F'$-class. So $A'$ is an analytic homogeneous set for $f'$. By Lemma~\ref{lem:reflection} there is then a Borel set $B \supseteq A'$ which is homogeneous for $f'$, and we are done.
\end{proof}

Observe that the statement that $E$ is prime to $F$ may be written as $E \rightarrow_B (E)^1_F$, and the statement that $E$ is Borel weakly compact may be written as $E \rightarrow_B (E)^2_2$. In fact, these are the only non-trivial Borel partition properties which are possible.
We can see that $n \geq 3$ will be impossible for non-trivial relations. When $n=2$ and $R=k$, we may replace $k$ by a larger integer, but $R=\omega$ will also fail for non-trivial $E$.

\begin{lem}[see \cite{Kechris}, Exercises 19.9 and 19.10]
Non-trivial partition relations with $n \geq 3$ or with $n=2$ and $F \geq \omega$ are impossible when $E$ has perfectly many classes. Namely:
\begin{enumerate} 
\item There is a clopen partition of $[2^{\omega}]^2$ into countably many clopen pieces which has no homogeneous set of size 3, i.e., $\bbR \not\rightarrow_B (3)^2_{\omega}$.
\item There is a partition of $[2^{\omega}]^3$ into two clopen pieces with no perfect homogeneous set, i.e., $\bbR \not\rightarrow_B (\bbR)^3_2$.
\end{enumerate}
\end{lem}
\begin{proof}
Let $\Delta(x,y)$ be the least $n$ with $x(n)\neq y(n)$ (or 0 when $x=y$). Identify $[2^{\omega}]^2$ with pairs $(x,y)$ with $x <_{\text{lex}} y$ and identify $[2^{\omega}]^3$ with triples $(x,y,z)$ with $x <_{\text{lex}} y <_{\text{lex}} z$ (where $<_{\text{lex}}$ is the lexicographical ordering on $2^{\omega}$).
For (1):  Set $P_n =\{ (x,y) : \Delta(x,y)=n \}$.
For (2):  Set $P_0 = \{(x,y,z) : \Delta(x,y) \leq \Delta(y,z)\}$ and $P_1=\{ (x,y,z) : \Delta(x,y)>\Delta(y,z)\}$.
\end{proof}

We can also consider how close an equivalence relation is to being prime.
We introduce a weakening of relative primeness.
\begin{dfn}
We say that $ F \leq_B E / R$ if for every Borel homomorphism $\varphi: E \rightarrow R$ there is $y$ so that $ F \leq_B E \upharpoonright \varphi^{-1} [y]_R$.
\end{dfn}
Thus $E$ is prime to $F$ if $E \leq_B E / F$.

\begin{dfn}
We say that $F$ is a \emph{Borel cofinality} for $E$ if for any Borel equivalence relation $R$, either $F \leq_B R$ or $F \leq_B E / R$.
\end{dfn}
In particular, $E$ is prime exactly when $E$ is a Borel cofinality for $E$. Note that any Borel cofinality for $E$ must be reducible to $E$. Also, any prime equivalence relation reducible to $E$ is a Borel cofinality for $E$.

\begin{lem}
\label{lem:prime-cofinal}
Let $F$ be a prime equivalence relation with $F \leq_B E$. Then $F$ is a Borel cofinality for $E$.
\end{lem}

\begin{proof}
Let $R$ be a Borel equivalence relation. If $F \not\leq_B R$, then $F$ is prime to $R$. Suppose $\varphi$ is a Borel homomoprhism from $E$ to $R$. Let $\rho$ be a reduction from $F$ to $E$, so that $\varphi \circ \rho$ is a homomorphism from $F$ to $R$. There is then a $y$ so that $F \leq_B F \upharpoonright (\varphi \circ \rho)^{-1} [y]_R$, so $F \leq_B E \upharpoonright \varphi^{-1} [y]_R$. 
\end{proof}

We have the following consequence of weaker partition relations for Borel cofinalities:

\begin{lem}
If $E \rightarrow_B (F)^2_2$ then for any Borel equivalence relation $R$, either $F \leq_B R$ or for any homomorphism $\varphi: E \rightarrow R$ we have that $F \leq_B E \upharpoonright \varphi^{-1}[y]_R$ for some $y$, i.e., $F \leq_B E / R$. Hence $F$ is a Borel cofinality for $E$.
\end{lem}

\begin{cor}
If $E \rightarrow_B (F)^2_2$ then whenever $E \leq_B R \times S$ we have either $F \leq_B R$ or $F \leq_B S$.
\end{cor}

Note that the infinite Ramsey Theorem shows that $\Delta(\omega)$ is Borel weakly compact, and the following theorem of Galvin (see \cite{Kechris} Theorem 19.7) shows the same for $\Delta(\bbR)$.

\begin{thm}[Galvin]
Let $X$ be a non-empty perfect Polish space and $[X]^2=P_0 \cup \cdots \cup P_{k-1}$ a partition where each $P_i$ has the Baire property. Then there is a Cantor set $C$ with $[C]^2 \subseteq P_i$ for some $i$.
\end{thm}

\begin{cor}
$\Delta(\bbR)$ is Borel weakly compact.
\end{cor}

This extends the following result (see, e.g., Theorem 19.1 of \cite{Kechris}), where $K(X)$ is the hyperspace of compact subsets of $X$:

\begin{thm}[Mycielski, Kuratowski]
\label{mycielski}
Let $X$ be a metrizable space, and let $n_i \in \omega \setminus \{0\}$ and $R_i \subseteq X^{n_i}$ be comeager for $i \in \bbN$. Then the set $\{ K \in K(X) : \forall i \ [K]^{n_i} \subseteq R_i \}$ is comeager in $K(X)$, where $[K]^n = \{(x_1.\ldots, x_n) \in K^n: \text{$x_i \neq x_j$ for $i \neq j$} \}$.
\end{thm}

Note also that $K_p(X) = \{ K \in K(X) : \text{$K$ is perfect} \}$ is a dense $G_{\delta}$ in $K(X)$, so for a Polish space $X$ and comeager sets $R_i$ as in the statement of the theorem there is a perfect set $K \subseteq X$ with $[K]^{n_i} \subseteq R_i$ for each $i$.

Conley has shown in Theorem 2.4 of \cite{conley} that Borel weak compactness also holds for $\bbE_0$.
\begin{thm}[Conley]
Suppose that $c : (2^{\omega})^2 \rightarrow 2$ is a symmetric, Baire measurable function. Then there exists a nonsmooth compact set $K$ such that $c$ is constant on $K^2 \setminus \bbE_0$.
\end{thm}

\begin{cor}
$\bbE_0$ is Borel weakly compact.
\end{cor}

We do not know if there are any other Borel weakly compact equivalence relations beyond these. We will see below that a natural candidate, $\bbF_2$, fails to be weakly compact, as does $\bbE_1$.

\begin{question}
Are there any Borel weakly compact equivalence relations above $\bbE_0$?
\end{question}

\begin{question} 
Given a benchmark equivalence relation $E$, for which $F$ do we have $E \rightarrow_B (F)^2_2$? What are the Borel cofinalities of $E$?
\end{question}

One can similarly define ``square bracket'' partition properties for equivalence relations. Other large cardinal partition relations on equivalence relations, such as J\'{o}nsson, Rowbotton, and Ramsey, have been considered in the context of AD by Jackson and Holshouser and others (see \cite{holshouser}). We may hope to find Borel analogues of other cardinal properties, such as tree properties and other partition relations. It is not clear, for instance, what a limit or strong limit should be in this context. Here the Friedman--Stanley jump operator might be used to give one interpretation of a limit. Similarly, we can ask if there are other equivalent properties to the above.

\section{Primeness and partition properties for ${\mathbb F}_2$}
\label{sec:f2prime}

We begin this section by developing some machinery for analyzing $\bbF_2$ due to Kanovei--Sabok--Zapletal, and using it to give a proof of their result that $\bbF_2$ is prime. These results, through Corollary~\ref{end-of-KSZ}, are derived from Section 6.1.3 of \cite{ksz}, where they are proved using Cohen forcing. We present them again here in a topological formulation which will be used in the second part of this section to analyze partition properties for $\bbF_2$ and to show that $\bbF_2$ is not Borel weakly compact. 
Throughout this section we will let $X= \left( 2^{\omega}\right)^{\omega}$. 
We let $S_{\infty}$ act on $X$ coordinatewise, i.e., $\pi \cdot x (k) = x(\pi^{-1}(k))$, and similarly for $(\pi_1,\ldots,\pi_n)$ acting on $X^n$. We assume all sets mentioned in definitions have the property of Baire.

Recall that the equivalence relation $\bbF_2$ is defined on $(2^{\omega})^{\omega}$ by setting $\bar{x} \mathrel{\bbF_2} \bar{y}$ iff $\{x_k : k \in \omega\} = \{y_k : k \in \omega\}$. It is straightforward to see that this is Borel bi-reducible with the orbit equivalence relation induced by $S_{\infty}$ acting coordinate-wise on $X$, and we will use both representation according to which is more convenient. We use  $2^{\omega}$ and $\bbR$ interchangably. We will also use set-theoretic notions like $\subseteq$, $\cup$, etc. even when referring to representatives of equivalence classes, i.e., countable sequences. For instance, $x \cup y$ can naturally be identified with $x \oplus y$ (defined below) and $x \subseteq y$ with $\{ x_n : n \in \omega\} \subseteq \{y_n : n \in \omega\}$. All of these operations and relations will be Borel on $\bbR^{\omega}$. We remark that $x \subsetneqq y$ will denote that $x$ is a proper subset of $y$.

\begin{dfn}
We say that $A \subseteq X^n$ is \emph{(finitely) invariant} if for any $\pi_1,\ldots, \pi_n \in S_{\infty}$ (with finite support) we have $(\pi_1,\ldots, \pi_n) \cdot A \subseteq A$.
Note that $A \subseteq X$ is invariant just in case it is $\bbF_2$-invariant.
\end{dfn}

\begin{lem} If $A \subseteq X^n$ is finitely invariant and non-meager then $A$ is comeager.
\end{lem}

\begin{proof}
If $A$ is non-meager, then $A$ is comeager in some basic open set $N_s$. We can find a sequence $\langle \pi_k : k \in \omega \rangle$ of finite support permutations in $S_{\infty}^n$ such that the sets $\pi_k (\text{range}(s))$ are pairwise disjoint. Then $A$ is comeager in $\bigcup_k \pi_k \cdot N_s$, which is open dense in $X^n$, so $A$ is comeager.
\end{proof}

\begin{dfn}
For $x_i \in X$ we define the following representations of unions:
\[ x_0 \oplus x_1 (j) = \begin{cases} x_0(k) \text{ if $j=2k$} \\ x_1(k) \text{ if $j=2k+1$} \end{cases} \]

\[\bigoplus_{i \leq n} x_i  = x_0 \oplus (x_1 \oplus (\cdots \oplus x_n) ) \]

\begin{align*}
\bigoplus_{i \in \omega} x_i (j) &= x_i(k) \text{ if $j+1 = 2^i(2k+1)$, i.e.,}\\
& \bigoplus_{i \in \omega} x_i = x_0 \oplus (x_1 \oplus (x_2 \oplus \cdots))  = \lim_n \bigoplus_{i \leq n} x_i .
\end{align*}
\end{dfn}

We begin by finding elements which are suitably generic with respect to finite unions.
\begin{lem}
\label{lem:invariantMK}
Let $R_i \subseteq X^{n_i}$ be comeager for $i \in \bbN$. Then there is a Cantor set $K \subseteq X$ such that:
\begin{enumerate}
\item For each $i$, $[K]^{n_i} \subseteq R_i$.
\item For each $i$, all $b_1,\ldots,b_{n_i} \in \bbN$, and all $x_{1,0}$, $\ldots$,  $x_{1,b_1}$, $\ldots$, $x_{n_i,0}$, $\ldots$, $x_{n_i, b_{n_i}}$ in $[K]^{n_i+b_1+\cdots+b_{n_i}}$ we have $(\bigoplus\limits_{j \leq b_1} x_{1,j},\ldots, \bigoplus\limits_{j \leq b_{n_i}} x_{n_i,j}) \in R_i$.
\item For each $i$ and $\pi_1,\ldots,\pi_{n_i}$ in $S_{\infty}$ with finite support we have all of the above containments with $(\pi_1,\ldots, \pi_{n_i}) \cdot R_i$ in place of $R_i$.
\item For $x \in K$ and $m \neq n$ we have $x(m) \neq x(n)$.
\item For $x, y \in K$ with $x \neq y$ and any $m$ and $n$ we have $x(m) \neq y(n)$.
\end{enumerate}
\end{lem}

\begin{proof}
We extend the collection of $R_i$'s to include all the comeager sets of the form $(\pi_1,\ldots, \pi_{n_i}) \cdot R_i$ for finite support permutations $\pi_1,\ldots,\pi_{n_i}$, and apply Theorem~\ref{mycielski} to the new countable collection of comeager sets $R_i$, together with  

\begin{gather*}
\{ x \in X : \forall n \neq m (x(n) \neq x(m)) \}, \\
\{(x,y) \in X^2 :  \forall m \forall n (x(m) \neq y(n))\}, \text{ and}   \\
 \{ (x_{1,0}, \ldots,  x_{1,b_1}, \ldots, x_{n_i,0}, \ldots, x_{n_i, b_{n_i}}) : (\bigoplus\limits_{j \leq b_1} x_{1,j},\ldots, \bigoplus\limits_{j \leq b_{n_i}} x_{n_i,j}) \in R_i \}, 
 \end{gather*}
noting that these sets are all comeager since permutations and maps of the form $(x_0, \ldots, x_{n}) \mapsto \bigoplus\limits_{i \leq n} x_i$ are homeomorphisms.
\end{proof}
Note that there are comeagerly many such $K$ in $K(X)$, and that condition (3) is automatic in the case that each $R_i$ is finitely invariant.  
The key will be to find such $K$ which satisfy an additional infinite genericity property in the case of invariant comeager sets. This is a topological formulation of Claim 6.29 of \cite{ksz}.

\begin{lem}[\textbf{Genericity Lemma}]
\label{lem:genericity}
Let $\langle R_i \subseteq X^{n_i} \rangle$ be a sequence of comeager sets. Then there is a Cantor set $K \subseteq X$ which satisfies the conclusion of Lemma~\ref{lem:invariantMK} for the sequence, with the following additional property: For each $R \subseteq X^n$ in the sequence, and distinct elements $\langle x_{1,j} : j \in \omega \rangle$, $\ldots$, $\langle x_{n,j} : j \in \omega \rangle$  of $K$, if we let $y_i = \bigoplus\limits_{j \in \omega} x_{i,j}$ for $1 \leq i \leq n$ then there is $\pi \in S_{\infty}$ such that $(\pi \cdot y_1, \ldots, \pi \cdot y_n) \in R$. In particular, when $R$ is invariant we have $(y_1,\ldots,y_n) \in R$.
\end{lem}

\begin{proof}
We replace each $R_i$ by a dense $G_{\delta}$ subset and find a Cantor set $K$ satisfying the conclusion of Lemma~\ref{lem:invariantMK} for the modified sequence; we will show that this $K$ satisfies the additional property for each $R=R_i$ in the sequence. For notational simplicity we will consider $R \subseteq X^2$ (no additional complications arise for larger $n$ as there is no interaction between coordinates). Let $R = \bigcap_n G_n$ with each $G_n$ open dense. Let $\langle x_j : j \in \omega \rangle$ and $\langle y_j : j \in \omega\rangle$ be distinct elements of $K$, and let $x = \bigoplus_{j \in \omega} x_j$ and $y = \bigoplus_{j \in \omega} y_j$. We will show there is $\pi \in S_{\infty}$ with $( \pi \cdot x, \pi \cdot y) \in R$.

We inductively build a sequence of finite support permutations $\pi_k$, an increasing sequence $n_k$ of natural numbers, and an increasing sequence of finite partial injections $\rho_k : n_k \rightarrow \omega$ with the following properties for all $k$:
\begin{enumerate}
\item $\rho_k \subseteq \rho_{k+1}$ and $k \subseteq \text{range}(\rho_k)$.
\item $\left( \pi_{k+1} \cdot \bigoplus_{j \leq k+1} x_j \right) \upharpoonright n_k = \left(\pi_k \cdot \bigoplus_{j \leq k} x_j \right) \upharpoonright n_k$ and
$\left( \pi_{k+1} \cdot \bigoplus_{j \leq k+1} y_j \right) \upharpoonright n_k = \left( \pi_k \cdot \bigoplus_{j \leq k} y_j \right) \upharpoonright n_k$.
\item For any $w$ and $z$ with $w\upharpoonright n_k = \left( \pi_k \cdot \bigoplus_{j \leq k} x_j  \right)\upharpoonright n_k$ and $z\upharpoonright n_k = \left(\pi_k \cdot \bigoplus_{j \leq k} y_j \right) \upharpoonright n_k$ we have $(w,z) \in G_k$.
\item If $\rho \in S_{\infty}$ extends $\rho_k$ then $\left( \rho^{-1} \cdot x \right) \upharpoonright n_k = \left( \pi_k \cdot \bigoplus_{j \leq k} x_j \right) \upharpoonright n_k$ and $\left( \rho^{-1} \cdot y \right)\upharpoonright n_k = \left( \pi_k \cdot \bigoplus_{j \leq k} y_j \right) \upharpoonright n_k$.
\end{enumerate}
At the end, we let $\rho = \bigcup_k \rho_k$, which is a bijection by condition (1), and set $\pi = \rho^{-1}$.
For each $k$, condition (4) gives $\left( \rho^{-1} \cdot x \right) \upharpoonright n_k = \left( \pi_k \cdot \bigoplus_{j \leq k} x_j \right) \upharpoonright n_k$ and $\left( \rho^{-1} \cdot y \right)\upharpoonright n_k = \left( \pi_k \cdot \bigoplus_{j \leq k} y_j \right) \upharpoonright n_k$, so condition (3) ensures $(\rho^{-1} \cdot x, \rho^{-1} \cdot y) \in G_k$, and hence we will have $(\pi \cdot x, \pi \cdot y) \in \bigcap_kG_k = R$ as desired.

For the construction, for $k \leq \omega$ let $d_k : \omega \rightarrow (k+1) \times \omega$ be the bijection so that $\bigoplus_{i \leq k} x_i (j) = x_{(d_k(j))_0}( (d_k(j))_1)$. Then conditions (2) and (4) will be ensured by requiring $d_{\omega} \circ \rho_k \upharpoonright n_k =  d_{k+1} \circ \pi_{k+1}^{-1} \upharpoonright n_k = d_k \circ \pi_k^{-1} \upharpoonright n_k$ for each $k$. 
When $k=0$, set $\pi_0=\text{id}$. Since $(x_0,y_0) \in G_0$, there are basic open sets $U$ and $V$ with $(x_0,y_0) \in U \times V \subseteq G_k$ and we can choose $n_0$ so that when $w \upharpoonright n_0 =x_0 \upharpoonright n_0$ and $z \upharpoonright n_0 = y_0 \upharpoonright n_0$ we have $w \in U$ and $z \in V$. Let $\rho_0(i)=2i$ for $i < n_0$.

Now, given $\pi_k$, $n_k$, and $\rho_k$ satisfying the above conditions, we define $\pi_{k+1}$, $n_{k+1}$, and $\rho_{k+1}$ as follows. We can choose a finite support permutation $\pi_{k+1}$ so that $\pi_{k+1}^{-1}(j) = d_{k+1}^{-1} \circ d_k \circ \rho_k (j)$ for $j < n_k$, ensuring condition (2). Since  $\left( \pi_{k+1} \cdot \bigoplus_{j \leq k+1} x_j ,   \pi_{k+1} \cdot \bigoplus_{j \leq k+1} y_j \right) \in G_{k+1}$, we can find $n'_{k+1} > n_k$ to satisfy condition (3) for $k+1$. We then define $\rho_{k+1}(j) = d_{\omega}^{-1} \circ d_{k+1} \circ \pi_{k+1}^{-1} (j)$ for $j < n'_{k+1}$, and extend the domain of $\rho_{k+1}$ to some $n_{k+1}\geq n'_{k+1}$ so that $k+1 \subseteq \text{range}(\rho_{k+1})$, ensuring conditions (1) and (4) for $k+1$ and completing the construction.
\end{proof}

A bit more care allows us to find a comeager set of permutations satisfying the conclusion.
We call a set $K$ satisfying the conclusion of the lemma \emph{good for $\langle R_i \rangle$}.

\begin{cor}[\textbf{Generic Embedding Lemma}]
\label{lem:generic-embedding}
Let $\langle R_i \subseteq X^{n_i} \rangle$ be a sequence of comeager invariant sets. Then there is a continuous embedding $\varphi : \bbF_2 \sqsubseteq_c \bbF_2$, which preserves set operations (in the sense that $\varphi(x \oplus y) \mathrel{\bbF_2} (\varphi(x) \oplus \varphi(y))$, etc.), so that for each $i$ and pairwise disjoint $x_1, \ldots, x_{n_i}$ in $\bbR^{\omega}$ we have $(\varphi(x_1),\ldots,\varphi(x_{n_i})) \in R_i$.
\end{cor}

\begin{proof}
Let $K$ be a good Cantor set from the Genericity Lemma, and choose a continuous embedding $g: 2^{\omega} \rightarrow K$. Let $\varphi(x)= \bigoplus_n g(x_n)$. The conditions of  goodness ensure that if $x_i$ and $x_j$ are disjoint then so are $g(x_i)$ and $g(x_j)$, so the Genericity Lemma applies.
\end{proof}

We call such a $\varphi$ the \emph{generic embedding for the sequence $\langle R_i \subseteq X^{n_i} \rangle$}.
We now make two simple observations about $\bbF_2$ and a coarser equivalence relation $F$:

\begin{lem}
Let $F$ be an equivalence relation with $\bbF_2 \subseteq F$ and suppose that the set $\{(x,y,z) : x \oplus y \FF x \oplus z \}$ is non-meager. Then $F$ has a comeager equivalence class.
\end{lem}
\begin{proof}
The given set is invariant since $\bbF_2 \subseteq F$, so if it is non-meager then it is comeager. By the Kuratowski--Ulam Theorem it suffices to show that the set $F$ is comeager in $X^2$. Using that $F$ is invariant, we have that the following are all comeager:
\begin{align*}
& \{ (x_0,x_1,y_0) : x_0 \oplus x_1 \FF x_0 \oplus y_0 \}, \\
& \{(x_0, y_0) : x_0 \oplus y_0 \FF y_0 \oplus x_0 \}, \\
& \{ (y_0, x_0, y_1) : y_0 \oplus x_0 \FF y_0 \oplus y_1 \} .
\end{align*}
From the transitivity of $F$ we then have that the set $\{(x_0,x_1,y_0,y_1) : x_0 \oplus x_1 \FF y_0 \oplus y_1 \}$ is comeager, and using that $(x_0,x_1) \mapsto x_0 \oplus x_1$ is a homeomorphism we have that $\{(x,y) : x \FF y\}$ is comeager as desired.
\end{proof}

\begin{lem}
Let $F$ be an equivalence relation with $\bbF_2 \subseteq F$ and suppose that the set $\{(x,y) : x \FF x \oplus y \}$ is non-meager. Then $F$ has a comeager equivalence class.
\end{lem}
\begin{proof}
This is similar to the previous lemma, using that  $\{(x,y) : x \FF x \oplus y\}$, $\{(x,y) : x \oplus y \FF y \oplus x\}$, and $\{(y,x) : y \FF y \oplus x\}$ will be comeager.
\end{proof}

We can now prove the main result concerning primeness of $\bbF_2$, which is Theorem 6.24 of \cite{ksz}.

\begin{thm}[Kanovei--Sabok--Zapletal]
\label{thm:ksz-main}
Let $F$ be an analytic equivalence relation with $\bbF_2 \subseteq F$. Then either $F$ has a comeager equivalence class or $\bbF_2 \sqsubseteq_c F \upharpoonright C$ for any $\bbF_2$-invariant comeager set $C$.
\end{thm}

\begin{proof}
Suppose that $F$ does not have a comeager class.
The above results then show that the following sets are comeager and invariant: $C$, $\{(x,y) : \neg x \FF y \}$, $\{ (x,y,z) : \neg x \oplus y \FF x \oplus z \}$, and $\{(x,y) : \neg x \FF x \oplus y\}$. Let  $\varphi$ be the generic embedding for these sets, so  that the range of $\varphi$ is contained in $C$.  If $x \mathrel{\bbF_2} y$ then $\varphi(x) \mathrel{\bbF_2} \varphi(y)$, so $\varphi(x) \FF \varphi(y)$. Suppose instead $\neg x \mathrel{\bbF_2} y$ and apply the Generic Embedding Lemma for  three possible cases. If $x \cap y = \emptyset$, then goodness for the set $\{(x,y) : \neg x \FF y \}$ shows $\neg \varphi(x) \FF \varphi(y)$. If $x \subsetneqq y$ (or vice versa) then goodness for $\{(x,y) : \neg x \FF x \oplus y\}$ (applied to $x$ and $y \setminus x$) again shows $\neg \varphi(x) \FF \varphi(y)$. Finally, in all other cases goodness for  $\{ (x,y,z) : \neg x \oplus y \FF x \oplus z \}$ (applied to $x \cap y$, $x \setminus y$, and $y \setminus x$) shows $\neg \varphi(x) \FF \varphi(y)$.
\end{proof}

Applying the theorem to $F=\bbF_2$ itself (whose classes are meager), we get:
\begin{cor}
\label{cor:F2comeager}
$\bbF_2$ maintains complexity on invariant comeager sets.
\end{cor}

We also have:
\begin{cor}
If $F$ is a Borel equivalence relation with $\bbF_2 \not\leq_B F$ then $\bbF_2$ is generically $F$-ergodic.
\end{cor}

From this and the characterization of prime equivalence relations from Lemma~\ref{lem:altprime} we then get our desired conclusion. In fact, in the case that $F$ does not have a comeager equivalence class, the range of $\varphi$ in the proof of Theorem~\ref{thm:ksz-main} will be a $F \setminus \bbF_2$-discrete set, so we have:
\begin{thm}
$\bbF_2$ is uniformly prime.
\end{thm}

Although, as noted earlier, consequences of primeness do not in general extend to infinite products and intersections, they do for $\bbF_2$.

\begin{cor}
If $\bbF_2 \leq_B \bigcap_n E_n$ then $\bbF_2 \leq_B E_n$ for some $n$.
\end{cor}

\begin{proof}
Let $f$ be a reduction from $\bbF_2$ to $\bigcap_n E_n$, so $f$ is a homomorphism from $\bbF_2$ to each $E_n$. If $\bbF_2$ were not reducible to any $E_n$, then for each $n$ there would be a comeager set $C_n$ so that the range of $f \upharpoonright C_n$ is contained in a single $E_n$ class. Then the range of $f$ on the comeager set $C=\bigcap_n C_n$ would be contained in a single $\bigcap_n E_n$ class, contradicting that $\bbF_2$ classes are meager.
\end{proof}

\begin{cor}
If $\bbF_2 \leq_B \prod_n E_n$ then $\bbF_2 \leq_B E_n$ for some $n$.
\end{cor}

\begin{proof}
Let $F_n = E_1 \times \cdots \times E_n \times \prod_{k>n} I(X) \sim_B E_1 \times \cdots \times E_n$, where $I(X) = X \times X$ is the trivial equivalence relation on $X$ with a single class. Then $\prod_n E_n = \bigcap_n F_n$, so $\bbF_2 \leq_B F_n$ for some $n$. Then $\bbF_2 \leq_B E_1 \times \cdots \times E_n$, so $\bbF_2 \leq_B E_i$ for some $i \leq n$ by primeness.
\end{proof}

Thus, for instance, since $\bbF_2 \not\leq_B \bbE_1^{\omega}$  we have:
\begin{cor}
\label{end-of-KSZ}
$\bbF_2$ is prime to $\bbE_1^{\omega}$.
\end{cor}

Unfortunately, attempts to generalize the above results about $\bbF_2$ to other equivalence relations encounter problems. First, note that for $\bbE_1$ we cannot use the meager ideal, as there are equivalence relations $F$ with $\bbE_1 \subseteq F$, every $F$-class meager, yet $\bbE_1 \not\leq_B F$. An example of such an $F$ is $\bbE_0^{\omega}$ ``with coordinates exchanged'' (i.e.,  $x \FF y$ if $\forall m \forall^{\infty} n \, x_n(m)=y_n(m)$), although in this case $\bbE_1$ will be reducible to its restriction to a single $F$-class.

Similarly, we cannot use the meager ideal for $\bbE_0^{\omega}$. Take $F$ to be $\bbE_0 \times I(2^{\omega})^{\bbN}$; then $\bbE_0^{\omega} \subseteq F$, every $F$-class is meager, yet $\bbE_0^{\omega} \not\leq_B F \sim_B \bbE_0$ (although again $\bbE_0^{\omega}$ will be reducible to its restriction to a single $F$-class).

Also, for $E$ such as $\bbE_1$ and $\bbE_0^{\omega}$ we do not have full $S_{\infty}$-invariance, but for $\pi \in S_{\infty}$ we do have $x \mathrel{E} y$ iff $\pi(x) \mathrel{E} \pi(y)$. We do not necessarily have $x \mathrel{F} y$ iff $\pi(x) \mathrel{F} \pi(y)$ for $E \subseteq F$, but might try to impose it somehow.

Now we turn to the analysis of possible partition properties for $\bbF_2$, and see that primeness cannot be improved to Borel weak compactness. Recall from Definition~\ref{dfn:partition} that $[\bbF_2]^2$ denotes the set of pairs of $\bbF_2$-inequivalent elements, and a partition (i.e., $\Delta(2)$-partition) of $[\bbF_2]^2$ is a symmetric invariant subset of $[\bbF_2]^2$. The following partition of $[\bbF_2]^2$ will be the linchpin to studying partition properties for $\bbF_2$:

\begin{dfn} Let $P_0$ be the following partition of $[\bbF_2]^2$: 
\[ (x,y) \in P_0 \ \Leftrightarrow\ x \not\subseteq y \ \wedge\ y \not\subseteq x .\]
\end{dfn}

We will make use of the following well-known observation:
\begin{lem}
\label{lem:orderable}
Let $E$ be an equivalence relation on a Polish space $X$. If $X/E$ admits a Borel linear ordering (in the sense that there is a Borel quasi-order on $X$ which is $E\times E$-invariant and so that the induced ordering on $X/E$ is linear) then $E$ is smooth.
\end{lem}

\begin{proof}
If not, $\bbE_0$ would be reducible to $E$ by the Generalized Glimm-Effros Dichotomy and we could pull back the ordering of $X/E$ to give a Borel quasi-ordering $\prec$ inducing a linear ordering of $2^{\omega}/E_0$. But there can be no such ordering (or even a function selecting one equivalence class from each unordered pair) or else, letting $x \mapsto \bar{x}$ be the homeomorphism of $2^{\omega}$ given by $\bar{x}(n)= 1 - x(n)$, we would have that the set $\{x : x \prec \bar{x} \}$ would be an $\bbE_0$-invariant set homeomorphic to its complement, contradicting the Topological 0-1 Law that any $\bbE_0$-invariant Baire-measurable set is either meager or comeager.
\end{proof}

We now see that homogeneous sets for the partition $P_0$ completely characterize partition relations for $\bbF_2$.
\begin{lem}
For any Borel equivalence relation $F$, $\bbF_2 \rightarrow_B (F)^2_2$ if and only if $F$ is reducible to $\bbF_2 \upharpoonright H_0$ for some Borel $H_0$ which is homogeneous for $P_0$.
\end{lem}

\begin{proof}
The left-to-right implication is immediate from the definitions, so suppose $H_0$ is homogeneous for $P_0$ and $\varphi_0$ is a reduction of $F$ to $\bbF_2 \upharpoonright H_0$. If $[H_0]^2 \cap P_0 = \emptyset$ then $\bbF_2 \upharpoonright H_0$ admits a Borel linear ordering induced by $\subseteq$ and is therefore smooth by the above observation. As Galvin's theorem  shows that there is a perfect homogeneous set for any partition of $[\bbF_2]^2$, we may then assume that $[H_0]^2 \subseteq P_0$. Also note that, following $\varphi_0$ with a map which replaces reals by pairwise disjoint countable sets and adds the same fixed set to everything, we may assume that for each $\bbF_2$-inequivalent $x$ and $y$ in $H_0$ we have $x \cap y$, $x \setminus y$, and $y \setminus x$ all infinite. Now let $P \subseteq [\bbF_2]^2$ be any Borel partition; we will show that $\bbF_2 \upharpoonright H_0 \leq_B \bbF_2 \upharpoonright H$ for some $H$ which is homogeneous for $P$.

Since the set $\{ (u,v,w) : u \oplus v \mathrel{P} u \oplus w\}$ is invariant, it is either meager or comeager. We will assume it is comeager and find $H$ with $[H]^2 \subseteq P$ (when it is meager we find $H$ with $[H]^2 \cap P = \emptyset$ in an identical fashion). Let  $\varphi (x)$ be the generic embedding for this set given by Corollary~\ref{lem:generic-embedding}. Then $\varphi$ is a reduction of $\bbF_2$ to itself; so letting $H$ be the range, we need only check that for $\bbF_2$-inequivalent $x$ and $y$ in $H_0$ we have $(\varphi(x),\varphi(y)) \in P$.  Letting $u = \varphi( x \cap y)$. $v = \varphi( x \setminus y)$, and $w=\varphi(y \setminus x)$, we have $\varphi(x) \mathrel{\bbF_2} u \oplus v$ and $\varphi(y) \mathrel{\bbF_2} u \oplus w$, and the Generic Embedding Lemma gives that $(u \oplus v , u \oplus w) \in P$ as desired.
\end{proof}

Hence analysis of partition properties for $\bbF_2$ reduces to studying $P_0$-homogeneous sets, i.e., antichains under containment. Characterizing such sets amounts to a definable version of Sperner's Theorem which characterizes the size of a maximal antichain under containment in the power set of a finite set. We do not have a complete characterization, but we can exhibit a lower bound.

\begin{thm}
\label{lem:Einfty-partition}
$\bbF_2 \rightarrow_B (E_{\infty}^{\omega})^2_2$.
\end{thm}

\begin{proof}
For simplicity, we will use $\omega^{\omega}$ as $\bbR$ here.
Let $\varphi$ be defined as
\[ \varphi(x) = \langle n \smallfrown g_m \cdot x_n : m,n \in \omega \rangle ,\]
where $\{g_n : n \in \omega\}$ is a countable group generating $E_{\infty}$, so $\varphi(x) = \{ n \smallfrown z : z \mathrel{E_{\infty}} x_n  \wedge n \in \omega\}$. This is easily a reduction from $E_{\infty}^{\omega}$ to $\bbF_2$. On the other hand, if there is an $n$ with $\neg x_n \mathrel{E_{\infty}} y_n$ then $[x_n]_{E_{\infty}}$ is disjoint from $[y_n]_{E_{\infty}}$ so $\varphi(x) \setminus \varphi(y)$ contains $\{ n \smallfrown z : z \mathrel{E_{\infty}} x_n \}$ whereas $\varphi(y) \setminus \varphi(x)$ contains $\{ n \smallfrown z : z \mathrel{E_{\infty}} y_n \}$, so $\varphi (x) \not\subseteq \varphi(y)$ and vice versa, so the range of $\varphi$ is homogeneous for $P_0$.
\end{proof}

We now show that $\bbF_2$ is not Borel weakly compact  by showing that there is no homogeneous set for $P_0$ of size $\bbF_2$. Zapletal has improved this result to show that the restriction of $\bbF_2$ to any $P_0$-homogeneous set must be pinned, but we present our original proof because it gives further insight into $\bbF_2$-invariant functions and possible generalizations to other equivalence relations. We begin with a technical lemma, showing that certain $\bbF_2$-invariant functions can not exist. We will use the following result,  which appears as Theorem 8.3.4 of \cite{gao}.

\begin{thm}[Harrington]
\label{thm:Harrington-separation}
Let $D \subseteq \omega^{\omega}$ be $\bPi^1_1$-complete and $A \subseteq \omega^{\omega}$ be $\bSigma^1_1$ with $D \subseteq A$. Then there is no Borel set $B \subseteq \omega^{\omega} \times \omega^{\omega}$ such that $D \times (A \setminus D) \subseteq B$ and $(A \setminus D) \times D \cap B = \emptyset$,
\end{thm}

The following lemma is a variation on Friedman's theorem on the non-existence of Borel diagonalizers. We give a proof based on Harrington's proof of Friedman's theorem, following the presentation given in Theorem 8.3.5 of \cite{gao}; we sketch a forcing proof below as well. We note that this lemma can be extended to other equivalence relations besides $\bbF_2$; however, the consequences concerning weak compactness do not follow for other relations.
\begin{lem}
\label{lem:F2diagonal}
There is no $\bbF_2$-invariant Borel function  $F$ with the following properties:
\begin{enumerate}
\item If $\neg x \mathrel{\bbF_2} y$ then $F(x) \not\subseteq F(y)$ and $F(y) \not\subseteq F(x)$.
\item If $x$ and $y$ are disjoint then $F(x)$ and $F(y)$ are disjoint.
\item If $x$ and $y$ are disjoint then $F(x)$ and $F(x \cup y)$ are disjoint.
\end{enumerate}
\end{lem}

\begin{proof}
Suppose there were such an $F$. First we can find $x_0$ so that $F(x_0) \cap x_0 = \emptyset$. To see this, note that condition (3) implies that we can not have $x \subseteq F(x)$ for all $x$, so let $x_1$ be such that $F(x_1) \cap x_1 \subsetneqq x_1$. If $F(x_1) \cap x_1 =\emptyset$ we are done; otherwise let $x_0 = F(x_1) \cap x_1 \subsetneqq x_1$, so that $F(x_0) \cap x_0 \subseteq F(x_0) \cap F(x_1) = \emptyset$. We now define $A \subseteq 2^{\omega} \times (R^{\omega})^{\omega}$ to be the set of all pairs $(x,f)$ so that $x \in \text{LO}$ codes a countable linear order and $f: \omega \rightarrow \bbR^{\omega}$ satisfies, for all $n \in \omega$:
 \begin{enumerate}
 \item $f(n) = F\left( F(x_0) \cup \bigcup\{ f(m) : m <_x n\}\right)$, where $<_x$ is the linear order of $\omega$ coded by $x$, and
 \item $f(n) \not\subseteq \bigcup \{ f(m) : m <_x n \}$.
 \end{enumerate}
 Then $A$ is a Borel set. We claim that for each $x$ coding a well-order there is an $f$ with $(x,f) \in A$.  Using transfinite recursion, define the sequence of sets $y_{\alpha} \in \bbR^{\omega}$ for $\alpha < \omega_1$ by
 \[ y_{\alpha} = F\left(F(x_0) \cup \bigcup_{\beta<\alpha} y_{\beta}\right) .\]
 We check by induction that for all $\alpha < \omega_1$, $y_{\alpha} \not\subseteq F(x_0) \cup \bigcup_{\beta < \alpha} y_{\beta}$. Note that $y_0= F(F(x_0))$ which is disjoint from $F(x_0)$ since $F(x_0)$ is disjoint from $x_0$. Since $F(x_0) \cap x_0 = \emptyset$, we have $\neg x_0 \mathrel{\bbF_2} F(x_0) \cup \bigcup_{\beta< \alpha} y_{\beta}$, and therefore $y_{\alpha} \not\subseteq F(x_0)$. For $\beta < \alpha$, by inductive assumption, $y_{\beta} \not\subseteq F(x_0) \cup \bigcup_{\lambda < \beta} y_{\lambda}$, so
 \[ F(x_0) \cup \bigcup_{\lambda < \beta} y_{\lambda} \subsetneqq F(x_0) \cup \bigcup_{\lambda < \alpha} y_{\lambda} , \]
 and thus $y_{\alpha} = F\left( F(x_0) \cup \bigcup_{\lambda < \alpha} y_{\lambda}\right)$ 
  is disjoint from $y_{\beta} = F\left( F(x_0) \cup \bigcup_{\lambda < \beta} y_{\lambda}\right)$. So $y_{\alpha}$ is not contained in $F(x_0)$ and is disjoint from $\bigcup_{\beta<\alpha} y_{\beta}$ and we have the claim. Transfinite recursion on $<_x$ now allows us to produce an $f$ with $(x,f) \in A$ for each $x \in \text{WO}$ in a uniform manner.
 
 Now let $D \subseteq A$ be given by $D=\{(x,f) \in A : x \in \text{WO}\}$, so that $D$ is $\bPi^1_1$-complete. We can now proceed mostly verbatim as in Theorem 8.3.5 of \cite{gao} to obtain a contradiction by producing a Borel set $B$ separating $D \times (A \setminus D)$ from $(A \setminus D) \times D$, and applying Theorem~\ref{thm:Harrington-separation}. For $u=(x,f)$ and $v=(y,g)$ in $A$, define a binary relation $R_{u,v} \subseteq \omega \times \omega$ by 
 \[n \mathrel{R_{u,v}} m \Leftrightarrow f(n) \mathrel{\bbF_2} g(m).\]
Condition (2) in the definition of $A$ ensures that  if $ n \mathrel{R_{u,v}} m_1$ and $ n \mathrel{R_{u,v}} m_2$ then $m_1 = m_2$. Similarly, if $ n_1 \mathrel{R_{u,v}} m$ and $ n_2 \mathrel{R_{u,v}} m$ then $n_1 = n_2$, so $R_{u,v}$ is a partial bijection between two subsets of $\omega$; we denote this partial bijection by $\varphi_{u,v}$.  We define $I_{u,v} \subseteq \text{dom}(\varphi_{u,v})$ and $J_{u,v} \subseteq \text{range}(\varphi_{u,v})$ by letting
 \begin{align*}
 n \in I_{u,v} & \Leftrightarrow \{ k : k \leq_x n\} \subseteq \text{dom}(\varphi_{u,v}) \\
 &\qquad \wedge \varphi_{u,v}(\{k: k \leq_x n\}) = \{l : l \leq_y \varphi_{u,v}(n)\} \\
 &\qquad \wedge \forall k, k' \leq_x n (k <_x k' \rightarrow \varphi_{u,v}(k) <_y \varphi_{u,v}(k')) 
 \end{align*}
 and $J_{u,v} = \varphi_{u,v}(I_{u,v})$. Finally, define $B \subseteq A \times A$ by setting
 \[ (u,v) \in B \Leftrightarrow I_{u,v} = \omega \vee ( \exists n I_{u,v}=\{ k : k <_x n\} \wedge J_{u,v} \neq \omega) . \]
 Then $B$ is Borel. To check that it separates $D \times (A \setminus D)$ from $(A \setminus D) \times D$, let $u=(x,f)$ and $v=(y,g)$ be in $A$. Suppose first that $x \in \text{WO}$ and $y \notin \text{WO}$; we will show $(u,v) \in B$. If $I_{u,v} = \omega$ we are done, so suppose not. Since $(I_{u,v},<_x)$ is a well-order and $\varphi_{u,v}$ is order-preserving we have $(J_{u,v},<_y)$ is also a well-order, so $J_{u,v} \neq \omega$ since $<_y$ is not a well-order; thus $(u,v) \in B$. Second, suppose $x \notin \text{WO}$ and $y \in \text{WO}$; we will show $(u,v) \notin B$. Suppose $(u,v) \in B$. Since $(I_{u,v},<_x)$ must be a well-order we have $I_{u,v} \neq \omega$, so there must be $n$ so that $I_{u,v} =\{k : k <_x n\}$ and $J_{u,v} \neq \omega$. As $(J_{u,v},<_y)$ is then a proper initial segment of $<_y$, there must be $m$ so that $J_{u,v} = \{ l : l <_y m\}$. For $k \in I_{u,v}$ we have $f(k) \mathrel{\bbF_2} g(\varphi_{u,v}(k))$, so that $\bigcup \{ f(k) : k <_x n\} = \bigcup\{g(l) : l <_y m\}$ and hence $f(n) \mathrel{\bbF_2} g(m)$. But then $n \mathrel{R_{u,v}} m$, i.e., $\varphi_{u,v}(n)=m$, contradicting that $n \notin I_{u,v}$. 
\end{proof}

We briefly sketch a forcing proof of the previous lemma; relevant facts may be found in \cite{zap1} and \cite{zap2}. Note that we may assume $F$ is injective, so that the set $Y = \bigcup \text{range}(F) = \{ a \in 2^{\omega}  : \exists x \exists n (a = F(x)_n) \}$ is Borel.  Let $\bbP = \text{Coll}(\omega,\bbR)$ be the forcing collapsing the reals of $V$ to a countable set, and let $\tau$ be a $\bbP$-name for a generic enumeration of $\bbR^V$, so $\tau[G]$ is a countable set of reals in the generic extension $V[G]$. Since mutually generic $G$ and $H$ will satisfy $\Vdash_{\bbP \times \bbP} \tau[G]=\tau[H]$ and $F$ is $\bbF_2$-invariant, we will have $\Vdash_{\bbP \times \bbP} F(\tau[G]) = F(\tau[H])$, so that $F(\tau)$ is $\bbF_2$-pinned, so there is a set $A \subseteq \bbR^V$ with $A \in V$ so that $\Vdash F(\tau) = \check{A}$. For all $x \in (\bbR^{\omega})^V$ we have $\Vdash \check{x} \subsetneqq \tau$, so $\Vdash F(\check{x}) \cap F(\tau) = \emptyset$, whence $\Vdash F(\check{x}) \cap A = \emptyset$. But then $Y \cap A = \emptyset$, whereas absoluteness gives $\Vdash F(\tau) \subseteq \check{Y}$, a contradiction.

Modifications of the previous lemma can also show the non-existence of other $\bbF_2$-invariant functions. For instance, by removing $F(x_0)$ in condition (1) of the definition of the set $A$ and in the definition of $y_{\alpha}$ in the proof we can obtain the following:

\begin{cor}
For any homomorphism $F$ from $\bbF_2$ to $\bbF_2$ there is $x$ with $F(x) \subseteq x$.
\end{cor}

Letting $F(x) = x \cup \{f(x) \}$ when $f: \bbR^{\omega} \rightarrow \bbR$ is an $\bbF_2$-invariant function, the corollary generalizes Friedman's original result.

We can now rule out Borel weak-compactness:

\begin{thm}
$\bbF_2 \not\rightarrow_B (\bbF_2)^2_2$, i.e., $\bbF_2$ is not Borel weakly compact.
\end{thm}

\begin{proof}
We will show that there is no homogeneous set $H$ for $P_0$ with $\bbF_2 \leq_B \bbF_2 \upharpoonright H$.
As observed earlier, if $H$ is such that $[H]^2 \cap P_0 = \emptyset$, then $\bbF_2 \upharpoonright H$ admits a linear ordering induced by $\subseteq$ and is therefore smooth. Hence we will assume there is a reduction $f$ from $\bbF_2$ to $\bbF_2 \upharpoonright H$ with $[H]^2 \subseteq P_0$ and derive a contradiction. As above, we may assume that for $\bbF_2$-inequivalent $x$ and $y$ we have both $f(x) \setminus f(y)$ and $f(y) \setminus f(x)$ infinite.

Fix $x \in \bbR^{\omega}$ and define the map $\varphi_x : \bbR^{\omega} \rightarrow 2^{\omega}$ by 
\[ \varphi_x(y)(n) = \begin{cases} 1 & \text{if $\neg \exists m f(x)(n)  = f(y)(m)$} \\ 0 & \text{otherwise,} \end{cases} ,\]
i.e., $\varphi_x(y)$ records $f(x) \setminus f(y)$ as a subset of $f(x)$. Since $\varphi_x$ is Borel and $\bbF_2$-invariant and $\bbF_2$ is generically $\Delta(2^{\omega})$-ergodic, we have that for each $x$ there is $z_x \in 2^{\omega}$ and a comeager set $C_x$ so that for all $y \in C_x$ we have $\varphi_x(y) = z_x$; hence for $y, z \in C_x$ we have $f(x) \setminus f(y) \mathrel{\bbF_2} f(x) \setminus f(z)$. The set 
\[ \{ (x,y,z) : f(x) \setminus f(y) \mathrel{\bbF_2} f(x) \setminus f(z) \} \]
is therefore invariant and comeager, so applying the Generic Embedding Lemma we can find an $\bbF_2$-invariant function $g_0: \bbR^{\omega} \rightarrow \bbR^{\omega}$ so that if $x$, $y$, and $z$ are disjoint we have $f(g_0(x))  \setminus f(g_0(y)) = f(g_0(x))  \setminus f(g_0(z))$;  call this set $w_{g_0(x)}$, which then depends only on $x$. Note that if $y$ and $z$ are each disjoint from $x$, we may find $w$ disjoint from all three, so that this conclusion does not require $y$ and $z$ be disjoint. Now set $\rho_0(x) = w_{g_0(x)}$ (enumerated in increasing order as a subset of $f(g_0(x))$) and observe that $\rho_0$ is Borel, since (noting that all countable sequences appearing are injective), we may calculate
\begin{align*}
\rho_0(x) &= \langle f(g_0(x))(n) : \exists y (x \cap y = \emptyset \wedge f(g_0(x))(n) \notin f(g_0(y)) ) \rangle \\
&= \langle f(g_0(x))(n) : \forall y (x \cap y = \emptyset \rightarrow f(g_0(x))(n) \notin f(g_0(y)) ) \rangle  .
\end{align*}
Hence $\rho_0$ is a Borel $\bbF_2$-invariant function with the property that if $x$ and $y$ are disjoint then $\rho_0(x)$ and $\rho_0(y)$ are disjoint. 

Applying the same approach to $f(x) \setminus f(x \cup y)$ instead of $f(x) \setminus f(y)$, we may find a Borel $\bbF_2$-invariant function $\rho_1$  with the property that if $x$ and $y$ are disjoint then $\rho_1(x)$ and $\rho_1(x \cup y)$ are disjoint. Let $F(x) = f(x) \times \rho_0(x) \times \rho_1(x)$, i.e., $F(x)_m = f(x)_{(m)_0} \oplus \rho_0(x)_{(m)_1} \oplus \rho_1(x)_{(m)_2}$, where $m=\langle (m)_0, (m)_1, (m)_2 \rangle$ is a bijection of $\omega$ with $\omega^3$. Then $F$ is a Borel reduction of $\bbF_2$ to $\bbF_2$ with the properties:
\begin{enumerate}
\item If $\neg x \mathrel{\bbF_2} y$ then $F(x) \not\subseteq F(y)$ and $F(y) \not\subseteq F(x)$.
\item If $x$ and $y$ are disjoint then $F(x)$ and $F(y)$ are disjoint.
\item If $x$ and $y$ are disjoint then $F(x)$ and $F(x \cup y)$ are disjoint.
\end{enumerate}
This contradicts Lemma~\ref{lem:F2diagonal}, and completes the proof.
\end{proof}

\begin{cor}
Uniform primeness does not imply Borel weak compactness.
\end{cor}

Zapletal has improved the above result, using forcing methods which we briefly sketch below. We have given the proof above using our original technique, as it seems likely to be generalizable to other equivalence relations, and the results about $\bbF_2$-invariant functions are of independent interest. In particular, it should be possible to extend the above result to show that $E^{+}$ is not Borel weakly compact for other equivalence relations $E$.

\begin{thm}[Zapletal]
If $H$ is homogeneous for $P_0$, then $\bbF_2 \upharpoonright H$ is pinned
\end{thm}

\begin{proof}[Sktech]
Since $\bbF_2$ restricted to a homogeneous set in the complement of $P_0$ must be smooth (and hence pinned) by Lemma~\ref{lem:orderable}, we need only consider an $P_0$-positive set $H$. Suppose that $\bbF_2$ restricted to $H$ is unpinned, and let $\tau$ be a non-trivial pinned name for an $\bbF_2$-class in $H$. Then there is an uncountable set of reals $A$ in $V$ so that $\text{Coll} \Vdash \tau= \check{A} \in H$. 
Let $M$ be a countable elementary substructure of $H_{\theta}$ for some sufficiently large $\theta$ with $A \in M$.
We have $A \cap M \in H$, so $M \vDash \text{Coll} \Vdash A^M \in H$, so $\text{Coll} \Vdash A \cap M \in H$.
Consider now a two-step iterated forcing $M[G_1][G_2]$. We have that $M[G_1][G_2]$ thinks that both $A \cap M$ and $A \cap M[G_1]$ are in $H$, but $A \cap M$ is a proper subset of $A \cap M[G_1]$, contradicting that $H$ is homogeneous for $P_0$. 
\end{proof}

As $\bbF_2$ is not pinned, this implies that $\bbF_2$ is not Borel weakly compact.

\begin{question}
Does $\bbF_2 \rightarrow_B (F)^2_2$ for every pinned $F$ with $F <_B \bbF_2$?
\end{question}

We can see that when $\bbF_2 \upharpoonright H$ is pinned, then there is no embedding of $\omega_1$ into $H$ under $\subseteq$, so we may hope to somehow decompose $H$ into $P_0$-homogeneous sets.
We observe that Lemma~\ref{lem:Einfty-partition} is not sharp, in that there are equivalence relations $F$ strictly above $E_{\infty}^{\omega}$  which satisfy $\bbF_2 \rightarrow_B (F)^2_2$. In particular, this holds for the relation $E_{\infty}^{[\bbZ]}$ considered in \cite{cc}. We mention another collection of pinned equivalence relations below $\bbF_2$ introduced in \cite{zap1}.

\begin{dfn} 
Let $\mathcal{G}$ be a Borel graph on $\bbR$. The set of \emph{$\mathcal{G}$-cliques}, $C_{\mathcal{G}}$, consists of those $x \in \bbR^{\omega}$ so that $(x_m,x_n) \in \mathcal{G}$ for all $m \neq n$. The equivalence relation $F_{\mathcal{G}}$ is then $\bbF_2 \upharpoonright C_{\mathcal{G}}$.
\end{dfn}

Then we have the following from \cite{zap1}:
\begin{lem}
$F_{\mathcal{G}}$ is pinned if and only if $\mathcal{G}$ has no uncountable cliques.
\end{lem}

\begin{question}
Does $\bbF_2 \rightarrow_B (F_{\mathcal{G}})^2_2$ for every Borel graph $\mathcal{G}$ with no uncountable cliques?
\end{question}

In the same way that the partition $P_0$ is central to $\bbF_2$, other natural partitions arise for equivalence relations which are symmetrizations of quasiorders, as considered, e.g., in \cite{rosendal}. For instance, the partition $\{ (x,y) : x \not\leq_T y \wedge y \not\leq_T x\}$ may provide insight into partition relations for $\equiv_T$, and similarly for $\equiv_A$ which gives a representation of $E_{\infty}$. In Section~\ref{section:E1} we will consider a partition which is similarly fundamental for $\bbE_1$.

\section{Iterates of the Friedman--Stanley jump}

We consider extensions of the above results about $\bbF_2$ to higher iterates of the Friedman--Stanley jump.

\begin{dfn}
For a Borel equivalence relation $E$ on $X$, the \emph{Friedman--Stanley jump of $E$}, denoted $E^{+}$, is defined on $X^{\omega}$ by setting $\bar{x} \mathrel{E^{+}} \bar{y}$ iff $\{ [x_n]_E : n \in \omega\} = \{[y_n]_E : n \in \omega\}$. Up to bireducibility, this is induced by the action of $S_{\infty}$ which permutes the coordinates of a sequence $\bar{x} \in X^{\omega}$.
\end{dfn}

In particular, $\bbF_2=\Delta(\bbR)^{+}$. Friedman--Stanley show in \cite{FS} that if $E$ has at least two classes then $E <_B E^{+}$. Kanovei--Sabok--Zapletal's result shows that $\Delta(\bbR)^{+}$ is prime to $\Delta(\bbR)$, and one may wonder if $E^{+}$ is prime to $E$ for general Borel $E$. This, however, is false.

\begin{lem}
\label{lem:doublejump}
If $F$ has perfectly many classes and $E$ is prime to $F^{+}$ then $E$ is prime to $F^{++}$.
\end{lem}

\begin{proof}
Note that $F$ has perfectly many classes iff $\bbF_2 \leq_B F^{+}$, so $E$ is prime to $\bbF_2$. 
Let $\varphi$ be a homomorphism from $E$ to $F^{++}$.
Let $\psi : F^{++} \rightarrow F^{+}$ be the homomorphism given by
\[ \psi(\{ \{ x_{i,n} : n \in \omega\} : i \in \omega \}) = \{x_{i,n} : i,n \in \omega \} \]
(using some pairing function to enumerate the countable set). Fix $y=\{y_i : i \in \omega\}$ and let $X = \psi^{-1}[y]_{F^{+}}$.
We claim that $F^{++} \upharpoonright X \leq_B \bbF_2$. To show this,
let $Y_i=[y_i]_F$ for $i \in \omega$. These are pairwise disjoint analytic sets (when $\neg y_i \EE y_j$) so by analytic separation we may find Borel sets $\widetilde{Y}_i \supseteq Y_i$ which are pairwise disjoint (when $\neg y_i \EE y_j$). On $X$, define the Borel function $f$ by
\[ f: \{ \{x_{i,n} : n \in \omega\} : i \in \omega \} \mapsto \{ \{j \in \omega : \exists n \ x_{i,n} \in \widetilde{Y}_j \} : i \in \omega\} .\]
This is a Borel reduction from $F^{++} \upharpoonright X$ to $\bbF_2$ (identifying $\mathcal{P}(\omega)$ with $2^{\omega}$).

Now, since $\psi \circ \varphi$ is a homomorphism from $E$ to $F^{+}$, there is $y$ so that $E \leq_B E \upharpoonright (\psi \circ \varphi )^{-1} [y]_{F^{+}}$.  Letting $X = \psi^{-1}[y]_{F^{+}}$, this shows that $E \upharpoonright \varphi^{-1}[X]$ is prime to $F^{+}$, and hence to $\bbF_2$, since $E \leq_B E \upharpoonright \varphi^{-1}[X]$. The claim shows that $F^{++} \upharpoonright X \leq_B \bbF_2$, so $E \upharpoonright \varphi^{-1}[X]$ is prime to $F^{++} \upharpoonright X$. Thus, there is $z \in X$ so that $E \upharpoonright \varphi^{-1}[X]$ (and hence $E$ itself) is reducible to $E \upharpoonright \varphi^{-1}[z]_{F^{++}}$. Hence $E$ is prime to $F^{++}$.
\end{proof}

Since no $E$ other than $\Delta(1)$ is prime to itself, this gives:
\begin{cor}
If $E$ has perfectly many classes, then $E^{++}$ is not prime to $E^{+}$. Hence if $E$ is a Borel equivalence relation with perfectly many classes then $E^{++}$ is not regular (and hence not prime).
\end{cor}

\begin{question}
Is it the case that if $\bbF_2 \leq_B E$ then $E^{+}$ is not regular?
\end{question}

Although the Friedman--Stanley jump does not always give an equivalence relation prime to $E$, we may ask if we can always find one.

\begin{question}
Does every equivalence relation $E$ have some $F$ which is prime to it? Is so, are there such $F$ of arbitrarily high Wadge degree? Is there a jump operator so that $J(E)$ is prime to $E$ for every Borel $E$?
\end{question}

We now consider iterations of the Friedman--Stanley jump.

\begin{dfn}
For $\alpha< \omega_1$, the equivalence relation $\bbF_{\alpha}$ is defined inductively by $\bbF_0 = \Delta(\omega)$ and for $\alpha>0$, $\bbF_{\alpha} = \left(\amalg_{\beta<\alpha} \bbF_{\beta}\right)^{+}$. Alternately, $\bbF_{\alpha}$ is the isomorphism relation on well-founded trees of rank at most $2+\alpha$.
\end{dfn}
Note that we could instead consider trees of rank less than $\alpha$, which makes a slight difference at limits.
The previous lemma in particular shows that  $\bbF_3=\bbF_2^{+}=\Delta(\bbR)^{++}$ is not prime to $\bbF_2$, and hence not regular. More generally:

\begin{cor}
If $E$ is prime to $\bbF_{\alpha}$ for $\alpha \geq 2$ then $E$ is prime to $\bbF_{\alpha+n}$ for all $n \in \omega$. Hence $\bbF_{\alpha+1}$ is not regular.
\end{cor}

\begin{proof}
Our definition implies that for $\lambda$ a limit, $\bbF_{\lambda} = \left(\amalg_{\beta<\lambda} \bbF_{\beta}\right)^{+}$, so each $\bbF_{\alpha}$ with $\alpha \geq 2$ is $F^{+}$ for some $F$ with perfectly many classes.
\end{proof}

In fact:
\begin{lem}
If $\alpha > 2$ and $\alpha$ is not of the form $\omega^{\beta}$ then $\bbF_{\alpha}$ is not regular.
\end{lem}

\begin{proof}
If $\alpha>2$ is not of the form $\omega^{\beta}$, then we can write $\alpha=\gamma + \delta$ with $\gamma < \alpha$ and $\delta + 1 < \alpha$. As in the proof of Lemma~\ref{lem:doublejump} we can define a homomorphism $\varphi$ from $\bbF_{\alpha}$ to $\bbF_{\gamma}$ by collapsing the top $\delta$ levels of the tree, so that $\bbF_{\alpha}$ restricted to the preimage of any class is reducible to $\bbF_{\delta+1}$.
\end{proof}

\begin{question}
Can $\bbF_{\omega^{\beta}}$ be regular? Borel weakly compact?
\end{question}

We can ask whether the earlier analysis that $\bbF_2$ is not Borel weakly compact can also be applied to $E^{+}$  for other Borel equivalence relations $E$.

\begin{question}
Does $E^{+} \rightarrow_B (E^{\omega})^2_2$ for suitably nice $E$?
Note that when $E=\bbF_{< \lambda}$ for $\lambda$ a limit, then $E^{+} \sim_B E^{\omega}$, which would imply that $\bbF_{\lambda}$ was Borel weakly compact. Thus this can not happen if $\lambda$ is not of the form $\omega^{\beta}$.
\end{question}

A weaker question:
\begin{question}
Given $E$, is there always an $F$ so that $F \rightarrow_B (E)^2_2$? Does $E^{+} \rightarrow_B (E)^2_2$?
\end{question}

\begin{question}
Suppose $F \leq_B E \leq_B F^{+}$ and $\rho : E \leq_B F^{+}$. Define $x \mathrel{P} y$ iff $\rho(x) \not\subseteq \rho(y)$ and $\rho(y) \not\subseteq \rho(x)$. What can we say about $P$-homogeneous sets?
\end{question}

Although $\bbF_3$ is not prime, we do get a weaker result from the primeness of $\bbF_2$ and Lemma~\ref{lem:prime-cofinal}:

\begin{cor}
$\bbF_2$ is a Borel cofinality for $\bbF_3$.
\end{cor}

\begin{question}
Which of the results about $\bbF_2$ from above extend to general $E^{+}$? Maintaining complexity on comeager sets? Failure of weak compactness?
\end{question}

\begin{question}
If $E$ is prime to $F$ and $E$ is prime to $\bbF_2$, is $E$ prime to $F^{+}$?
\end{question}

\begin{question}
If $E$ is prime to $\bbF_2$, is $E$ prime to $\bbF_{\alpha}$ for all $\alpha < \omega_1$? 
\end{question}

This is true if primeness is replaced by generic ergodicity, as in the case of turbulent actions. Here the difficulty arises at limit ordinals, which are trivially handled for generic ergodicity.

\begin{question}
If $E$ is prime to $\bbF_{\alpha}$ for all $\alpha<\omega_1$, is $E$ prime to $\cong_{\text{graph}}$? 
\end{question}

This would imply $\cong_{\text{graph}}$ is not regular. This is related to the question of Friedman--Stanley of whether $E$ is Borel-complete if all $\bbF_{\alpha}$ are reducible to $E$.

\begin{question}
Is graph isomorphism prime? 
\end{question}

\section{${\mathbb E}_1$ and equivalence relations induced by ideals}
\label{section:E1}

As noted earlier, a natural candidate for primeness is $\bbE_1$.

\begin{conjecture}
$\bbE_1$ is a prime equivalence relation.
\end{conjecture}

Kechris and Louveau have conjectured that for any Borel equivalence relation $F$, either $\bbE_1 \leq_B F$ or $F$ is reducible to some orbit equivalence relation $E_G^X$. If true, this would imply the above conjecture since $\bbE_1$ is prime to any $F$ which is reducible to some $E_G^X$. A special case of this conjecture, due to Hjorth, is that $\bbE_1$ is reducible to any treeable equivalence relation which is not essentially countable.
The question of whether or not $\bbE_1$ is prime seems much more approachable. We might consider whether the two conjectures are equivalent.

\begin{question}
If $\bbE_1$ is prime to $E$, is $E \leq_B E_G^X$ for some orbit equivalence relation $E_G^X$?
If $E$ is treeable and $\bbE_1$ is prime to $E$, is $E$ essentially countable?
\end{question}

We do not know whether $\bbE_1$ is prime; however, we will show that $\bbE_1$ fails to be Borel weakly compact. We also obtain some partial results and possible approaches to the question of primeness. Recall from our earlier characterization that $\bbE_1$ is prime iff whenever $\bbE_1 \subseteq F$ is such that $\bbE_1 \upharpoonright [x]_F <_B \bbE_1$ for all $x$, then $\bbE_1 \leq_B F$. The condition that $\bbE_1 \upharpoonright [x]_F <_B \bbE_1$ is equivalent to saying that $\bbE_1 \upharpoonright [x]_F$ is essentially hyperfinite or that $\bbE_1 \upharpoonright [x]_F$ is essentially countable.
We use the following parameterized version of a result from Chapter 11 of \cite{kanovei}:

\begin{thm}[Kanovei]
Let $X \subseteq (2^{\omega})^{\omega}$ be $\Delta^1_1(z)$ for some parameter $z$. Then exactly one of the following holds:
\begin{enumerate}
\item For all $x \in X$, $\forall^{\infty} n \ x_n \in \Delta^1_1(z,x_{> n})$, where $x_{>n} = \langle x_m \rangle_{m>n}$.
\item $\bbE_1 \leq_B \bbE_1 \upharpoonright X$.
\end{enumerate}
\end{thm}

Thus, to establish that $\bbE_1$ is prime, it is sufficient to show the following:
Whenever $\bbE_1 \subseteq F$ and for all $(x,y) \in F$ we have
$\forall^{\infty} n\ x_n \in \Delta^1_1(y, x_{> n})$, then $\bbE_1 \leq_B F$. This characterization motivates the following partition for considering Borel weak compactness of $\bbE_1$:

\begin{dfn}
Let $P$ be the following partition on $[\bbE_1]^2$:
\[ P(x,y) \Leftrightarrow \forall m \forall^{\infty}n (x_n \in \Delta^1_1(y_{>m},x_{>n}) \wedge y_n \in \Delta^1_1(x_{>m},y_{>n})) .\]
\end{dfn}
This is symmetric and $\bbE_1\times \bbE_1$-invariant. Note, though, that it is $\Pi^1_1$ and not Borel. We make a few simple observations about this partition.

\begin{lem}
Suppose $F$ is a $\Delta^1_1$ equivalence relation with $\bbE_1 \subseteq F$ and $\bbE_1 \upharpoonright [x]_F <_B \bbE_1$ for all $x$. Then $F \subseteq P$. 
\end{lem}

\begin{proof}
Suppose $x \FF y$. For each $m$, let $z_m$ be the result of replacing the first $m$ columns of $y$ by 0's so that $z_m \mathrel{\bbE_1} y$ (and hence $z_m \FF y)$ and $z_m \equiv_T y_{>m}$. Then $X=[x]_F=[y]_F=[z_m]_F$ is $\Delta_1^1(z_m)$, so $\forall^{\infty} n\ x_n \in \Delta^1_1(z_m,x_{>n}) = \Delta^1_1(y_{>m},x_{>n})$. Similarly,  $\forall^{\infty} n\ y_n \in \Delta^1_1(x_{>m},y_{>n})$.
\end{proof}

\begin{lem}
If $[H]^2_{\bbE_1} \subseteq P$ then $\bbE_1 \upharpoonright H <_B \bbE_1$.
\end{lem}

\begin{proof}
Let $H$ be $\Delta^1_1(z)$ for some $z$, and fix any $y \in H$, so $H$ is also $\Delta^1_1(z \oplus y)$. For any $x \in H$ we have $\forall^{\infty} n \ x_n \in \Delta^1_1(y,x_{>n}) \subseteq \Delta^1_1(z \oplus y, x_{>n})$, so $\bbE_1 \not\leq_B \bbE_1 \upharpoonright H$.
\end{proof}

\begin{lem}
 If $[H]^2_{\bbE_1} \cap P = \emptyset$ then $\bbE_1 \upharpoonright H \leq_B F$ for any Borel equivalence relation $F$ with $\bbE_1 \subseteq F \subseteq P$. 
 \end{lem}
 
 \begin{proof}
 If $[H]^2_{\bbE_1} \cap P = \emptyset$ and $F \subseteq P$, then $H$ is $F \setminus \bbE_1$-discrete.
 \end{proof}
 
 Hence $\bbE_1$ would be uniformly prime if there were a homogeneous set of size $\bbE_1$ (necessarily satisfying $[H]^2_{\bbE_1} \cap P = \emptyset$) for this partition $P$. This, however, turns out to be false. In fact there is a Borel sub-partition of $P$ with no homogeneous set of size $\bbE_1$.
 
 \begin{dfn}
 Let $\widetilde{P}$ be the following partition on $[\bbE_1]^2$:
\[ \widetilde{P}(x,y) \Leftrightarrow \forall m \forall^{\infty}n (x_n \leq_T y_{>m} \oplus x_{>n} \wedge y_n \leq_T x_{>m} \oplus y_{>n}) .\]
 \end{dfn}
 
 Since $\widetilde{P} \subseteq P$, if $[H]^2_{\bbE_1} \subseteq \widetilde{P}$ then $\bbE_1 \upharpoonright H <_B \bbE_1$. We will also rule out a large homogeneous set disjoint from $\widetilde{P}$. The following builds on the existence of a square coding function for the ideal generated by reverse cubes in \cite{ksz} (see Theorem 7.34 of \cite{ksz}), but is complicated slightly by the fact that we need to encode a set of columns on which to do our coding.

\begin{lem}
If $[H]^2_{\bbE_1} \cap \widetilde{P} = \emptyset$ then $\bbE_1 \upharpoonright H <_B \bbE_1$.
\end{lem}

\begin{proof}
Suppose $\bbE_1 \leq_B \bbE_1 \upharpoonright H$; we will show there are $\tilde{x}, \tilde{y} \in H$ with $\neg \tilde{x} \mathrel{\bbE_1} \tilde{y}$ and $\widetilde{P}(\tilde{x},\tilde{y})$. By Theorem 7.14 of \cite{ksz}, there is an infinite $s \subseteq \omega$ so that $H$ contains an \emph{$s$-cube}, i.e., the range of a continuous map $f: \left( 2^{\omega} \right)^{\omega} \rightarrow \left( 2^{\omega} \right)^{\omega}$ with the properties that 
\begin{enumerate}
\item if $x(i) \neq y(i)$, then $f(x)(\pi(j)) \neq f(y)(\pi(j))$ for all $j \leq i$, and
\item if $x(j)=y(j)$ for all $j \geq i$, then $f(x)(n)=f(y)(n)$ for all $n \geq \pi(i)$, 
\end{enumerate}
where $\pi: \omega \rightarrow s$ is an increasing enumeration of $s$. We will produce increasing sequences $t_n$, $k_n$, and $d_n$ from $\omega$, and elements $x_n,y_n \in \left( 2^{\omega} \right)^{\omega}$ with the following properties:
\begin{enumerate}
\item $x_{n+1} \upharpoonright \left( 2^{k_n} \right)^{k_n} = x_n \upharpoonright \left( 2^{k_n} \right)^{k_n}$ and $y_{n+1} \upharpoonright \left( 2^{k_n} \right)^{k_n} = y_n \upharpoonright \left( 2^{k_n} \right)^{k_n}$.
\item Whenever $x \upharpoonright \left( 2^{k_n} \right)^{k_n} = x_n \upharpoonright \left( 2^{k_n} \right)^{k_n}$ and $y \upharpoonright \left( 2^{k_n} \right)^{k_n} = y_n \upharpoonright \left( 2^{k_n} \right)^{k_n}$ then $f(x) \upharpoonright \left( 2^{d_n+1} \right)^{t_n+1} = f(x_n) \upharpoonright \left( 2^{d_n+1} \right)^{t_n+1}$ and $f(y) \upharpoonright \left( 2^{d_n+1} \right)^{t_n+1} = f(y_n) \upharpoonright \left( 2^{d_n+1} \right)^{t_n+1}$.
\item $f(x_n)(t_n) \upharpoonright d_n = f(y_n)(t_n) \upharpoonright d_n$ and $f(x_n)(t_n)(d_n) \neq f(y_n)(t_n)(d_n)$.
\item $f(x_n)(t_n)(d_n)= z(n)$, where $z(n) = \begin{cases} f(x_n)(m)(k) & \text{if $n=\langle m,k,0 \rangle$} \\ f(y_n)(m)(k)  & \text{ if $n=\langle m,k,1\rangle$} \end{cases}$, 

\noindent
where $n = \langle m,k,i \rangle$ is a bijection from $\omega$ to $\omega \times \omega \times 2$ with $m,k < n$ for $n\neq 0$ (and $0=\langle 0,0,0\rangle$).
\item For all $m \geq n+1$, $t_{n+1}$ is the least $k \geq d_n$ so that $f(x_m)(k) \neq f(y_m)(k)$.
\item If $i$ is least such that $d_n \leq \pi(i)$, then $t_{n+1} \leq \pi(i)$ and $d_{n+1} \geq \pi(i+1)$.
\item If $\pi(i) \geq t_n$, then $x_n(m) = y_n(m)$ for all $m \geq i+1$.
\end{enumerate}

At the end, the limits $x_{\infty} = \lim_n x_n$ and $y_{\infty}=\lim_n y_n$ exist, and we let $\tilde{x}=f(x_{\infty})$ ,  $\tilde{y}=f(y_{\infty})$, and $t = \langle t_n \rangle \subseteq \omega$. Then $\tilde{x}$ and $\tilde{y}$ are in $H$, and $\neg \tilde{x} \mathrel{\bbE_1} \tilde{y}$ by condition (3). We have $t \leq_T \tilde{x}_{>m}\oplus \tilde{y}_{>m}$ for all $m$, and $\tilde{x}_n \leq_T t \oplus \tilde{y}_{>m}\oplus \tilde{x}_{>n}$ and $\tilde{y}_n \leq_T t \oplus \tilde{x}_{>m}\oplus \tilde{y}_{>n}$ for all $m$ and $n$, so $\widetilde{P}(\tilde{x},\tilde{y})$ as desired.

For the construction,  let $t_0 = \pi(0)$ be the least element of $s$. Choose any $x$, and choose $k$ large enough so that $f(x)(0)(0)$ is determined by $x \upharpoonright \left( 2^{k} \right)^{k}$. By the properties of $f$, we can then find $x_0$ and $y_0$ so that $x_0 \upharpoonright \left( 2^{k} \right)^{k} = y_0 \upharpoonright \left( 2^{k} \right)^{k} = x \upharpoonright \left( 2^{k} \right)^{k}$, $f(x_0)(t_0) \upharpoonright \pi(1) = f(y_0)(t_0) \upharpoonright \pi(1)$, $x_0(0) \neq y_0(0)$ (so $f(x_0)(t_0) \neq f(y_0)(t_0)$), $x_0(m)=y_0(m)$ for $m \geq 1$, and $f(x_0)(t_0)(d_0)= f(x_0)(0)(0)=f(y_0)(0)(0)$, where $d_0$ is the least $d$ such that $f(x_0)(t_0)(d) \neq f(y_0)(t_0)(d)$. Choose $k_0$ large enough so that $f(x_0) \upharpoonright \left( 2^{d_0+1} \right)^{t_0+1}$ is determined by $x_0 \upharpoonright \left( 2^{k_0} \right)^{k_0}$ and similarly for $y_0$.

Given $t_n$, $d_n$, $k_n$, $x_n$, and $y_n$ meeting the conditions, let $i$ be least so that $d_n \leq \pi(i)$.  Let $t_{n+1}$ be the least $k \geq d_n$ meeting the following condition: 
\begin{quote}
For every $d$, there are $x$ and $y$ with $x \upharpoonright \left( 2^{k_n} \right)^{k_n} = x_n \upharpoonright \left( 2^{k_n} \right)^{k_n}$ and $y \upharpoonright \left( 2^{k_n} \right)^{k_n} = y_n \upharpoonright \left( 2^{k_n} \right)^{k_n}$, $f(x)(k) \upharpoonright d = f(y)(k) \upharpoonright d$, $f(x)(k) \neq f(y)(k)$, $f(x)(j) = f(y)(j)$ for all $d_n \leq j < k$, and $x(m)=y(m)$ for all $m \geq i+1$.
\end{quote}
We claim that such a $k$ exists, and satisfies $k \leq \pi(i)$. Suppose the condition does not hold for any $k$ with $d_n \leq k < \pi(i)$.  Then for each such $k$ there is a bound $b_k$ so that if $x$ and $y$ meeting the rest of the condition satisfy $f(x)(k) \upharpoonright b_k = f(y)(k) \upharpoonright b_k$ then $f(x)(k)=f(y)(k)$. We can take large enough initial segments of $x_n$ and $y_n$ (for which $f(x_n)$ and $f(y_n)$ agree on such columns $k$) to fix identical values  of $f(x)(k) \upharpoonright b_k = f(y)(k) \upharpoonright b_k$, thus ensuring all future extensions  satisfy $f(x)(k)=f(y)(k)$. Now for any $d$, we can find further extensions $x$ and $y$ so that $f(x)(y)(\pi(i)) \upharpoonright d = f(y)(\pi(i)) \upharpoonright d$ but $x(i)\neq y(i)$ so that $f(x)(\pi(i)) \neq f(y)(\pi(i))$, so that $k=\pi(i)$ satisfies the condition.

Take $d= \pi(i+1)$, and find $x$ and $y$ witnessing the above condition for $k=t_{n+1}$. Let $d_{n+1}\geq d$ be the least $j$ so that $f(x)(t_{n+1})(j) \neq f(y)(t_{n+1})(j)$. We claim that we can modify $x$ and $y$ if necessary to find $x_{n+1}$ and $y_{n+1}$ additionally satisfying  $f(x_{n+1})(t_{n+1})(d_{n+1})= z(n+1)$ (noting that $z(n+1)$ has already been determined, as this is either $f(x_{n+1})(m)(k)$ or $f(y_{n+1})(m)(k)$ with $m,k \leq n$, which must agree with $f(x_{n})(m)(k)$ and $f(y_{n})(m)(k)$, respectively, by condition (2)).  If $f(x)(t_{n+1})(d_{n+1})= z(n+1)$ we may take $x_{n+1}=x$ and $y_{n+1}=y$, so suppose instead that $f(y)(t_{n+1})(d_{n+1}) = z(n+1)$. Consider large enough initial segments of $x_n$ and $y_n$ (and hence of $x$ and $y$) which agree on columns $\geq i-1$ and force agreement of $f(x)$ and $f(y)$ on columns $d_n \leq k < t_{n+1}$ and on column $t_{n+1}$ up to $d$. Note that these agree on all columns beyond $t_{n+1}$. Let $x_{n+1}$ and $y_{n+1}$ be the results of interchanging the additional coordinates of $x$ and $y$. Then $x_{n+1}(j) = y(j)$ for $j \geq i-1$, so $f(x_{n+1})(m)  = f(y)(m)$ for $m \geq\pi(i-1)$, and also $f(y_{n+1})(m)=f(x)(m)$ for $m \geq \pi(i-1)$, so $d_{n+1}$ is still the first disagreement in column $t_{n+1}$ and $f(x_{n+1})(t_{n+1})(d_{n+1})= y(t_{n+1})(d_{n+1})=z(n+1)$ as required. This completes the construction of stage $n+1$, and hence the proof.
\end{proof}

Hence we have:

\begin{thm}
$\bbE_1$ is not Borel weakly compact.
\end{thm}

We thus have a complete characterization of Borel partition properties for $\bbE_1$:

\begin{cor}
A Borel equivalence relation satisfies $\bbE_1 \rightarrow_B (F)^2_2$ if and only if $F \leq_B \bbE_0$.
\end{cor}

Note that $\widetilde{P}$ is far from transitive, and transitivity of an equivalence relation $F$ is likely to be a crucial distinction necessary to establish uniform primeness vs. weak compactness for $\bbE_1$, as it was for $\bbF_2$. We mention a question somewhat related to primeness:
\begin{question}
If $\bbE_1 \leq_B F^{+}$ is $\bbE_1 \leq_B F$?
\end{question}

Although we do not know how to establish that $\bbE_1 \leq_B F$ whenever $\bbE_1 \subseteq F$ and $\bbE \upharpoonright [x]_F$ is essentially countable for all $x$, we can prove a weaker result if we assume that $\bbE_1 \upharpoonright [x]_F$ is countable for each $x$, not just essentially countable:

\begin{lem}
\label{lem:ctbleoverE1}
If $F$ is a $\bSigma^1_1$ equivalence relation with $\bbE_1 \subseteq F$ and $\bbE_1 \upharpoonright [x]_F$ is countable for all $x$, then $\bbE_1 \sqsubseteq_c F$.
\end{lem}

We will show that this in fact true for a large class of equivalence relations induced by ideals, such as $\bbE_0^{\omega}$, which we introduce.

\begin{dfn}
Let $P$ be a countable set, and $\ideal$ an ideal on $P$. We define the equivalence relation $E_{\ideal}$ on $2^P$ by $x \mathrel{E_{\ideal}} y$ iff $\{p : x(p) \neq y(p)\} \in \ideal$.
\end{dfn}

We introduce some specific ideals we will use.

\begin{dfn}
The \emph{empty ideal} on $\omega$ is $0=\{\emptyset\}$. The \emph{finitary ideal} or \emph{Fr\'{e}chet ideal} on $\omega$ is ${\text{FIN}} = \{ A \subseteq \omega: \text{$A$ is finite} \}$. The \emph{summable ideal} on $\omega$ is $\ideal_s = \{A \subseteq \omega : \sum_{n \in A} \frac{1}{n+1}  < \infty \}$.
\end{dfn}

\begin{dfn}
Given ideals $\ideal$ on $P$ and $\ideal[J]$ on $Q$, the product ideal $\ideal \times \ideal[J]$ is defined on $P \times Q$ by setting $A \in \ideal \times \ideal[J]$ iff $\{ x : A_x \notin \ideal[J]\} \in \ideal$.
\end{dfn}

For instance, $\bbE_0 = E_{\text{FIN}}$, $\bbE_1 = E_{\text{FIN} \times 0}$, $\bbE_0^{\omega} = E_{0 \times \text{FIN}}$, $\bbE_1^{\omega} = E_{0 \times \text{FIN} \times 0}$, and $\bbE_2 = E_{\ideal_s}$.

\begin{dfn}
Let $E$ be an equivalence relation on $X = 2^P$, where $P$ is a countable set. For $x, z \in 2^P$ we write $x \upharpoonright z$ for the element given by $x \upharpoonright z (p) = \min (x(p),z(p))$. In particular, for $x,y \in X$ we have $x \upharpoonright z = y \upharpoonright z$ iff $\forall p (z(p)=1 \rightarrow x(p)=y(p))$.
\end{dfn}

\begin{dfn}
For $E$ an equivalence relation on $X = 2^P$, we say that $E$ \emph{generically separates classes} if for all $x,y \in X$ we have 
\[ \forall^{\ast} z \forall x' \forall y' (x \EE x' \wedge y \EE y' \wedge x' \upharpoonright z = y' \upharpoonright z \rightarrow x' \EE y').\]
\end{dfn}

\begin{dfn}
For $E \subseteq F$ on $X = 2^P$, we say $E$ \emph{generically separates classes within $F$} if for every $w \in X$ we have
\[ \forall^{\ast} z  \forall x (x \FF w \wedge x \upharpoonright z = w \upharpoonright z \rightarrow x \EE w) .\]
\end{dfn}

\begin{lem}
\label{lem:countable-sep}
If $E$ generically separates classes and $F$ is of countable index over $E$, then $E$ generically separates classes within $F$.
\end{lem}

\begin{proof}
Fix $w$ and let $\{ x_i : i \in \omega \}$ enumerate representatives of the $E$-classes contained in $[w]_F$, so $[w]_F = \bigcup_i [z_i]_E$. For each $i$ and $j$ there is then a comeager set $C_{i,j}$ such that 
\[ \forall z \in C_{i,j} \forall x \forall y (x_i \EE x \wedge x_j \EE y \wedge x \upharpoonright z = y \upharpoonright z \rightarrow x \EE y).\]
Then for all $z$ in the comeager set $\bigcap_{i,j} C_{i,j}$ we have the desired conclusion since every $x \in [w]_F$ is in some $[x_i]_E$.
\end{proof}

\begin{dfn}
We say that $E$ on $X= 2^P$ \emph{generically maintains complexity on sections} if for a comeager set of $z$ and $\alpha$ we have that $E \sqsubseteq_c E \upharpoonright C$ for any $C$ comeager in $X_{z,\alpha}$, where $X_{z,\alpha} = \{ x \in X : x \upharpoonright z = \alpha \upharpoonright z \}$.
\end{dfn}
Note that when $z$ is co-infinite (which happens generically, and is the intention) then $X_{z,\alpha}$ is a Silver cube homeomorphic to $2^P$.
We can take as a clopen basis for the topology of $2^P$ all sets of the form $N_s= \{ x : x \upharpoonright \text{dom}(s) = s\}$, where $s \in 2^{< P}$.

\begin{lem}
\label{lem:hypermycielski}
Let $C \subseteq 2^P$ be comeager. Then $\forall^{\ast}z \forall^{\ast} \alpha\; X_{z,\alpha} \subseteq C$.
\end{lem}

\begin{proof}
Let $C \supseteq \bigcap_i G_i$ where each $G_i$ is open dense. Let $H_i = \{(z,\alpha) : X_{z,\alpha} \subseteq G_i \}$. It will suffice to show that each $H_i$ is open dense in $2^P \times 2^P$. Note that since $X_{z,\alpha}$ is compact and the map $(z,\alpha) \mapsto X_{z,\alpha}$ is continuous from $2^P$ to $K(2^P)$, we have that $H_i$ is open (see Section 4F of \cite{Kechris} for properties of the space $K(X)$). To see that it is dense, let $N_{s}\times N_{t}$ be any basic open set in $2^P \times 2^P$; by shrinking we may assume $\text{dom}(s)=\text{dom}(t)$. Let $z = s \smallfrown 1^{P \setminus \text{dom}(s)}$, so $z \in N_{s}$. Let $\{ c_j : j < N \}$ enumerate $\{c \in 2^{\text{dom}(s)} : c \upharpoonright s = t \upharpoonright s\}$, so that for any $(z,\alpha) \in N_{s}\times N_{t}$ and any $x \in X_{z,\alpha}$ we have that $c_j \sqsubseteq x$ for some $c_j$. 
Since $G_i$ is dense, there is some $d_0$ such that $N_{c_0 \smallfrown d_0} \subseteq G_i$. Likewise there is $d_1 \sqsupseteq d_0$ such that $N_{c_1 \smallfrown d_1} \subseteq G_i$, and for each $j<N$ there is $d_j \sqsupseteq d_{j-1}$ such that $N_{c_j \smallfrown d_j} \subseteq G_i$. Now if $\alpha$ is any extension of $t \smallfrown d_{N-1}$ we will have that $(z,\alpha) \in N_{s}\times N_{t}$ and $X_{z,\alpha} \subseteq G_i$, so $H_i$ is dense.
\end{proof}

\begin{lem}
If $E$ generically maintains complexity on sections, then $E$ maintains complexity on non-meager sets.
\end{lem}

\begin{proof}
Let $C \subseteq 2^P$ be non-meager, so $C$ is comeager in some neighborhood $N_s$ determined by $s \in 2^{< P}$. If $z(n) = 1$ and $\alpha(n)=s(n)$ for all $n < |s|$, then $X_{z, \alpha} \subseteq N_s$, so there is a non-meager set of $z$ and $\alpha$ with $C$ comeager in $X_{z, \alpha}$. Hence there is such a pair with $E \sqsubseteq_c E \upharpoonright C \cap X_{z, \alpha}$.
\end{proof}

\begin{lem}
\label{lem:sep-embed}
If $E$ is $\bPi^1_1$, $F$ is $\bSigma^1_1$, $E \subseteq F$, $E$ generically maintains complexity on sections, and $E$ generically separates classes within $F$, then $E \sqsubseteq_c F$.
\end{lem}

\begin{proof}
We have that for all $w$ there is a comeager set of $z$ with $(w,z)$ in the $\bPi^1_1$ set 
\[ A = \{ (w,z) :  \forall x (x \FF w \wedge x \upharpoonright z = w \upharpoonright z \rightarrow x \EE w) \}, \]
so by the Kuratowski--Ulam Theorem we have that $\forall^{\ast}z (\text{$A^z$ is comeager})$, and since $E$ generically maintains complexity on sections we may choose $z$ and a comeager set $C \subseteq X$ such that $(w,z) \in A$ for all $w \in C$ and for a comeager set of $\alpha$ we have $E \sqsubseteq_c E \upharpoonright D$ for any comeager $D \subseteq X_{z,\alpha}$.
We can also choose $z$ so that for $i \in 2$ the sets $P_i  = \{p \in P : z(p)=i\}$ are both infinite. Let $\varphi$ be the natural homeomorphism from $2^{P_0} \times 2^{P_1}$ to $2^P$; note that $\varphi(\beta,\gamma) \in X_{z,\alpha}$ for any $\alpha$ with $\alpha \upharpoonright P_1 = \gamma$. We have that $\{(\beta,\gamma) : \varphi(\beta,\gamma) \in C\}$ is comeager, so there is $\gamma_0$ such that $\forall^{\ast} \beta \in 2^{P_0} (\varphi(\beta,\gamma_0) \in C)$, and such that there is $\alpha$ with $\alpha \upharpoonright P_1 = \gamma_0$ for which $E \sqsubseteq_c E \upharpoonright D$ for any comeager $D \subseteq X_{z,\alpha}$. But then the map $\beta \mapsto \varphi(\beta,\gamma_0)$ is a homeomorphism from $2^{P_0}$ to $X_{z,\alpha}$, so the set $D = C \cap X_{z,\alpha}$ is comeager in $X_{z,\alpha}$. There is then a continuous embedding $\psi$ from $E$ to $E \upharpoonright D$. Since $D \subseteq C$ and $x \upharpoonright z = y \upharpoonright z = \alpha \upharpoonright z = \gamma_0$ 
for all $x,y \in X_{z,\alpha}$, we have that if $\psi(x) \FF \psi(y)$ then $\psi(x) \EE \psi(y)$; thus $\psi$ is a continuous embedding from $E$ to $F$.
\end{proof}

It remains to show that particular equivalence relations generically maintain complexity on sections, and generically separate classes within certain larger equivalence relations. We will see that this holds for $E_{\ideal}$ when $\ideal$ is suitably robust.

\begin{dfn}
We say that an ideal $\ideal$ on $P \times Q$ is \emph{vertically invariant} if for every function $f: P \times Q \rightarrow P \times Q$ satisfying $(f(p,q))_0=p$ and $f(p,q_1)\neq f(p,q_2)$ for $q_1 \neq q_2$ we have that for every $A \subseteq P \times Q$, $A \in \ideal$ iff $f[A] \in \ideal$.
\end{dfn}

An equivalent property of ideals was defined in \cite{clm}, which used the terminology \emph{determined by cardinalities on vertical sections}.

\begin{lem}
Any ideal of the form $\ideal[J] \times 0$ or $\ideal[J] \times \text{FIN}$ is vertically invariant.
\end{lem}

\begin{proof}
This is immediate since membership in such ideals depends only on the cardinality in each section, and the admissible functions do not change this.
\end{proof}

Note that $\bbE_0 = E_{\text{FIN}}$, $\bbE_1 = E_{\text{FIN} \times 0}$, $\bbE_0^{\omega} = E_{0 \times \text{FIN}}$, and $\bbE_1^{\omega} = E_{0 \times \text{FIN} \times 0}$, so these equivalence relations are generated by vertically invariant ideals (where we can  think of $\text{FIN}$ as $\ideal[J] \times \text{FIN}$ where $\ideal[J]$ is the trivial ideal on a one element set).

\begin{lem}
\label{lem:visections}
Let $\ideal$ be a vertically invariant ideal on $P \times Q$. Then $E_{\ideal}$ on $2^{P \times Q}$ generically maintains complexity on sections.
\end{lem}

\begin{proof}
The set $U=\{ z \in 2^{P \times Q} : \forall p \exists^{\infty} q ( z(p,q)=0) \}$ is comeager. For such $z$ and any $\alpha$, let $C \subseteq X_{z,\alpha}$ be comeager. A straightforward relativization of Lemma~\ref{lem:hypermycielski} to $X_{z,\alpha}$ shows that for a relatively comeager set of $z' \supseteq z$ and $\alpha'$ with $\alpha' \upharpoonright z = \alpha \upharpoonright z$ we have $X_{z',\alpha'} \subseteq C$. Choose some such $z' \in U$ and any such $\alpha'$. Let $e(p,\cdot)$ be a bijection for each $p$ from $Q$ to the set of $q'$ such that $z'(p,q')=0$. Define $\varphi: X \rightarrow X_{z',\alpha'}$ by $\varphi(x)(p,e(p,q)) = x(p,q)$ and $\varphi(x)(p,q)=\alpha'(p,q)$ otherwise. As the map $(p,q) \mapsto (p,e(p,q))$ satisfies the hypotheses of vertical invariance, this is an embedding of $E_{\ideal}$ into $E_{\ideal} \upharpoonright C$.
\end{proof}

\begin{lem}
\label{lem:separating-condition}
For an ideal $\ideal$, $E_{\ideal}$ generically separates classes iff for any $x$ and $y$ with $\neg x \mathrel{E_{\ideal}} y$ the set $\{ z :\neg x \upharpoonright z \mathrel{E_{\ideal}} y \upharpoonright z\}$ is comeager.
\end{lem}

\begin{proof}
Suppose $\neg x \upharpoonright z  \mathrel{E_{\ideal}} y \upharpoonright z$, and $x' \mathrel{E_{\ideal}} x$ and $y' \mathrel{E_{\ideal}} y$. Then $x' \upharpoonright z \mathrel{E_{\ideal}} x \upharpoonright z$ and $y' \upharpoonright z \mathrel{E_{\ideal}} y \upharpoonright z$, so if $\neg x' \mathrel{E_{\ideal}} y'$ then $\neg x' \upharpoonright z \mathrel{E_{\ideal}} y' \upharpoonright z$; in particular $x' \upharpoonright z \neq y' \upharpoonright z$, so $z$ satisfies the separation condition for $x$ and $y$. Conversely, suppose $z$ satisfies this condition for $x$ and $y$ with $\neg x \mathrel{E_{\ideal}} y$. Let $y' = y \Delta (x \upharpoonright z \Delta y \upharpoonright z)$, so $x \upharpoonright z = y' \upharpoonright z$. If $y \mathrel{E_{\ideal}} y'$ then $x \mathrel{E_{\ideal}} y'$, contradicting that $\neg x \mathrel{E_{\ideal}} y$. Hence $\neg y \mathrel{E_{\ideal}} y'$, so $\neg x \upharpoonright z \mathrel{E_{\ideal}} y \upharpoonright z$. 
\end{proof}

\begin{lem}
$E_{\text{FIN}}=\bbE_0$ generically separates classes.
\end{lem}

\begin{proof}
Let $x$ and $y$ be given with $\{n : x(n) \neq y(n) \} \notin \text{FIN}$. Then the  set $\{ z : \exists^{\infty} n (x(n) \neq y(n) \wedge z(n)=1) \}$ is comeager and satisfies the condition of Lemma~\ref{lem:separating-condition} for generically separating classes.
\end{proof}

\begin{lem}
If $\ideal[J]$ is an ideal such that $E_{\ideal[J]}$ generically separates classes, then so does $E_{\ideal[J] \times 0}$.
\end{lem}

\begin{proof}
Let $\neg x \mathrel{E_{\ideal[J] \times 0}} y$; then there is $A \notin \ideal[J]$ such that for all $p \in A$ there is $q_p$ with $x(p,q_p) \neq y(p,q_p)$; set $q_p=0$ for $p \notin A$. For $w \in 2^{P \times Q}$, let $\widetilde{w} \in 2^P$ be given by $\widetilde{w}(p)=w(p,q_p)$. Then $\neg \tilde{x} \mathrel{E_{\ideal[J]}} \tilde{y}$, so by Lemma~\ref{lem:separating-condition} there is a comeager set $C_0$ of $z \in 2^P$ for which $\neg \tilde{x} \upharpoonright z \mathrel{E_{\ideal[J]}} \tilde{y} \upharpoonright z$. Define $C \subseteq 2^{P \times Q}$ by $C = \{ z : \tilde{z} \in C_0\}$. Then $C$ is comeager, and if $x \upharpoonright z \mathrel{E_{\ideal[J] \times 0}} y \upharpoonright z$ then $\tilde{x} \upharpoonright \tilde{z} \mathrel{E_{\ideal[J]}} \tilde{y} \upharpoonright \tilde{z}$, so $E_{\ideal[J] \times 0}$ generically separates classes by Lemma~\ref{lem:separating-condition}.  
\end{proof}

\begin{lem}
If $\ideal$ is any ideal and $\ideal[J]$ is an ideal such that $E_{\ideal[J]}$ generically separates classes, then so does $E_{\ideal \times \ideal[J]}$. 
\end{lem}

\begin{proof}
Let $\neg x \mathrel{E_{\ideal \times \ideal[J]}} y$, so there is a set $A \notin \ideal$ such that for all $p \in A$ we have $\neg x_p \mathrel{E_{\ideal[J]}} y_p$. Then by Lemma~\ref{lem:separating-condition} for each $p \in A$ there is a comeager set $C_p$ of $z$ for which $\neg x_p \upharpoonright z \mathrel{E_{\ideal[J]}}  y_p \upharpoonright z$. Let $C = \{ z : \forall p \in A (z_p \in C_p) \}$. Then $C$ is comeager and for any $z \in C$ and any $p \in A$ we have $(x \upharpoonright z)_p = x_p \upharpoonright z_p$ and $y_p \upharpoonright z_p = (y \upharpoonright z)_p$ so $\neg (x \upharpoonright z)_p \mathrel{E_{\ideal[J]}} (y \upharpoonright z)_p$; hence $\neg x \upharpoonright z \mathrel{E_{\ideal \times \ideal[J]}} y \upharpoonright z$. Thus $E_{\ideal \times \ideal[J]}$ generically separates classes by Lemma~\ref{lem:separating-condition}.
\end{proof}

Although the summable ideal $\ideal_s$ generating $\bbE_2$ is not vertically invariant, we can still establish the relevant properties for $\bbE_2=E_{\ideal_s}$.

\begin{lem}
$\bbE_2$ generically maintains complexity on sections.
\end{lem}

\begin{proof}
As $\ideal_s$ is a free ideal with the Baire property, it is meager; hence the set $U= \{z : \bar{z} \notin \ideal_s\}$ is comeager, where $\bar{z}(n) = 1 - z(n)$. Fix $z \in U$ and any $\alpha$, and let $C \subseteq X_{z,\alpha}$ be comeager. As in Lemma~\ref{lem:visections} we can find $z' \in U$ and $\alpha'$ such that $X_{z',\alpha'} \subseteq C$. As $\bar{z'} \notin \ideal_s$, we can find disjoint finite sets $A_k \subseteq \omega$ such that $z'(n)=0$ for all $n \in A_k$ and $| \sum \{\frac{1}{n+1}: n \in A_k\} - \frac{1}{k+1}| < 2^{-k}$. The map $\varphi$ given by $\varphi(x)(n) = x(k)$ if $n \in A_k$ and $\varphi(x)(n) = \alpha'(n)$ otherwise is then an embedding of $\bbE_2$ into $\bbE_2 \upharpoonright X_{z',\alpha'}$.  
\end{proof}

\begin{lem}
$\bbE_2$ generically separates classes.
\end{lem}

\begin{proof}
Let $\neg x \mathrel{\bbE_2} y$. Then $\neg x \upharpoonright z \mathrel{\bbE_2} y \upharpoonright z$ iff $\forall k \sum \{ \frac{1}{n+1} : x \upharpoonright z(n) \neq y \upharpoonright z(n) \} > k$. The inside condition is evidently open dense for each $k$, so the set of such $z$ is comeager and hence $\bbE_2$ generically separates classes by Lemma~\ref{lem:separating-condition}.
\end{proof}

Putting this all together we arrive at:

\begin{thm}
Let $\ideal$ be one of the following: $\text{FIN}$,  $\text{FIN} \times 0$, $\ideal_s$, $\ideal_s \times 0$, or a $\bPi_1^1$ ideal of the form $\ideal[J] \times \text{FIN}$ or $\ideal[J] \times \ideal_s \times 0$ or $\ideal[J] \times \text{FIN} \times 0$. Then for any $\bSigma_1^1$ equivalence relation $F$ of countable index over $E_{\ideal}$ we have $E_{\ideal} \sqsubseteq_c F$.
\end{thm}

\begin{proof}
All such ideals (excepting $\ideal_s$) are vertically invariant so the corresponding equivalence relations generically maintain complexity on sections, and all generically separate classes. Hence they generically separate classes within $F$ by Lemma~\ref{lem:countable-sep}, and so $E_{\ideal} \sqsubseteq_c F$ by Lemma~\ref{lem:sep-embed}.
\end{proof}

This establishes Lemma~\ref{lem:ctbleoverE1}, as well as the following:
\begin{cor}
If $F$ is a $\bSigma_1^1$ equivalence relation of countable index over $\bbE_0^{\omega}$, then $\bbE_0^{\omega} \sqsubseteq_c F$.
If $F$ is a $\bSigma_1^1$ equivalence relation of countable index over $\bbE_1^{\omega}$, then $\bbE_1^{\omega} \sqsubseteq_c F$.
If $F$ is a $\bSigma_1^1$ equivalence relation of countable index over $\bbE_2$, then $\bbE_2 \sqsubseteq_c F$.
\end{cor}

\begin{question}
Does this hold for any Polishable ideal?
\end{question}

In establishing that the above equivalence relations generically maintain complexity on sections, we have also established that they maintain complexity on non-meager sets.

\begin{thm}
\label{thm:non-meager-sections}
Let $\ideal$ be one of the following: $\text{FIN}$, $\ideal_s$, or a $\bPi_1^1$ ideal of the form $\ideal[J] \times 0$ or $\ideal[J] \times \text{FIN}$. Then $E_{\ideal}$ maintains complexity on non-meager sets. In particular, this holds for $\bbE_0$, $\bbE_1$, $\bbE_2$, $\bbE_0^{\omega}$, and $\bbE_1^{\omega}$.
\end{thm}

We may hope to apply similar techniques to $F$ which are smooth over $E$ or essentially countable over $E$ in order to prove primeness for $\bbE_1$ or $\bbE_0^{\omega}$.
Note that rather than requiring a comeager set of $z$ in the above properties, it would suffice to have a set which was large with respect to some ideal which preserves Baire category.

\section{Non-dichotomy results}

Earlier work of Clemens--Lecomte--Miller in \cite{clm} has ruled out the possibility of certain global dichotomies. In particular, Theorem 2 of \cite{clm} established the following:

\begin{thm}[Clemens--Lecomte--Miller]
If $\Gamma$ is a Borel Wadge class containing $\bSigma^0_2$, then there is no minimum non-potentially $\Gamma$ Borel equivalence relation $E$.
\end{thm}
Here an equivalence relation $E$ is minimum non-potentially $\Gamma$ if for every Borel equivalence relation $F$, either $F$ is reducible to some equivalence relation in $\Gamma$, or $E \leq_B F$. We have that $\Delta(\bbR)$ is a minimum non-potentially $\bDelta^0_1$ equivalence relation, and $\bbE_0$ is a minimum non-potentially $\bPi^0_2$ equivalence relation, but these are the only such.
We might hope to be able to instead use primeness to characterize potentially $\Gamma$ equivalence relations for more complex $\Gamma$; however, we will see that this, too, is impossible. We briefly summarize some concepts and results from \cite{clm}.

\begin{prop} There is a treeing $T$ inducing a $K_{\sigma}$ equivalence relation $E_{[T]}$ with the following properties:
\begin{enumerate} 
\item If $\ideal \supseteq \text{FIN}$ is an ideal on $\omega$, then $E_{\ideal}^{\ast} = E_{\ideal} \cap E_{[T]}$ is treeable, with a treeing given by $T \cap E_{\ideal}$.
\item  $E_{\ideal}^{\ast} \subseteq E_{[T]}$, and the restriction of $E_{\ideal}^{\ast}$ to a single $E_{[T]}$-class is smooth.
\item If $\ideal$ is {vertically invariant} and $\ideal \notin \Gamma$ for a Wadge class $\Gamma$, then $E_{\ideal}^{\ast} \notin \text{pot}(\Gamma)$.
\item There are vertically invariant $\ideal$ of arbitrarily high Wadge degree, so the class of all $E_{\ideal}^{\ast}$ has elements of cofinal potential Wadge degree.
\end{enumerate}
\end{prop}

From this we get:

\begin{lem}
If $E$ is prime to $E_{[T]}$, then $E$ is prime to $E_{\ideal}^{\ast}$ for all $\ideal$.
\end{lem}

\begin{proof}
Since $\Delta(\bbR) \leq_B E_{[T]}$, $E$ is prime to any smooth equivalence relation by Lemma~\ref{lem:basic} (2). Then by property (2) of $E_{[T]}$ and part (3) of Lemma~\ref{lem:basic}, we have that $E$ is prime to $E_{\ideal}^{\ast}$.
\end{proof}

For instance, since $\bbE_0^{\omega}$ is prime to any $F_{\sigma}$ equivalence relation:

\begin{cor}
$\bbE_0^{\omega}$ is prime to $E_{\ideal}^{\ast}$ for all $\ideal$.
\end{cor}

Note that for a Wadge class $\Gamma$, we can clearly not have an equivalence relation $E$ such that $F \in \text{pot}(\Gamma)$ iff $E$ is not prime to $F$; the above lemma also immediately rules out a converse approach for non-trivial $\Gamma$:

\begin{thm}
If $\Gamma$ is a Wadge class containing $\bSigma^0_2$, then there is no equivalence relation $E$ such that for any Borel equivalence relation $F$, $F \in \text{pot}(\Gamma)$ iff $E$ is prime to $F$.
\end{thm}

\begin{proof}
Since $\bSigma^0_2 \subseteq \Gamma$, we have $E_{[T]} \in \Gamma$, hence $E$ would be prime to $E_{[T]}$. But then $E$ is prime to all $E_{\ideal}^{\ast}$, which are not all in $\text{pot}(\Gamma)$, a contradiction.
\end{proof}

We can ask whether any sufficiently complicated equivalence relation has this property.

\begin{question}
Let $E$ be a Borel equivalence relation which is not essentially hyperfinite. Are there $F$ of arbitrarily high potential Wadge degree such that $E$ is prime to $F$?
\end{question}
This would strengthen Theorem~6.2 of \cite{clm} that for any Borel equivalence relation $E$ which is not essentially hyperfinite there are $F$ of arbitrarily high potential Wadge degree such that $E$ is incompatible with $F$ under Borel reducibility.


\providecommand{\bysame}{\leavevmode\hbox to3em{\hrulefill}\thinspace}
\providecommand{\MR}{\relax\ifhmode\unskip\space\fi MR }
\providecommand{\MRhref}[2]{%
  \href{http://www.ams.org/mathscinet-getitem?mr=#1}{#2}
}
\providecommand{\href}[2]{#2}

\end{document}